\documentclass{article}
\author{}
\usepackage{todonotes}

\usepackage{amsmath, amscd, graphicx}
\usepackage{amssymb}
\usepackage{amsthm}
\usepackage{amsrefs}
\usepackage{cancel}
\usepackage[utf8]{inputenc}
\usepackage[T1]{fontenc}
\usepackage{amsmath}
\usepackage{amsfonts}
\usepackage{amssymb}
\usepackage[version=4]{mhchem}
\usepackage{stmaryrd}
\usepackage{bbold}
\usepackage{cases}
\usepackage{hyperref}
\usepackage{lipsum}

\DeclareSymbolFont{largesymbol}{OMX}{yhex}{m}{n}
\newtheorem{theo}{Theorem}[section]
\newtheorem{lemm}[theo]{Lemma}

\newtheorem{coro}[theo]{Corollary}
\newtheorem{prop}[theo]{Proposition}

\newtheorem{rema}{Remark}[section]
\numberwithin{equation}{section}

\begin{document}
	\title{ Ill-posedness of incompressible  Kelvin-Helmholtz problem with transverse magnetic field }
	
	\author{ Binqiang Xie $^{\dag}$, Boling Guo$^{\ddag}$ , Bin Zhao $^{\ddag}$ \\[10pt]
	\small {$^\dag$ School of Mathematics and Statistics,}\\
	\small { Guangdong University of Technology, Guangzhou,   510006, China}
	\\
	\small {$^\ddag$  Institute of Applied Physics and Computational Mathematics, Beijing, 100088,  China}
}

\footnotetext{\\ E-mail addresses: \it xbqmath@gdut.edu.cn(B.Q. Xie), \it  gbl@iapcm.ac.cn(B.L. Guo),  \it zhaobin2017math@163.com(B. Zhao).}

\date{}

\date{}
	
	\maketitle
 	\begin{abstract}
	In this paper, we prove the linear and  nonlinear  ill-posedness of the  well-known Kelvin-Helmholtz problem of the  incompressible ideal  magnetohydrodynamics (MHD) equations with transverse magnetic field. Our proof rigorously verifies  that  "the development of the Kelvin-Helmholtz instability, in the direction of the streaming, is uninfluenced by the presence of the magnetic field in the transverse direction" which was proposed by S. Chandrasekhar' book named by   Hydrodynamic and Hydromagnetic stability \cite{Chandrasekhar}. 
	\vspace*{5pt}

\vspace*{5pt}
\end{abstract}
\section{Introduction}
\subsection{Eulerian formulation}
\quad   This paper concerns the incompressible, inviscid Kelvin-Helmholtz problem for incompressible MHD fluids in an domain $\Omega:= \mathbb{T}^{2} \times (-1,1)$. More precisely, we consider two distinct inviscid incompressible, immiscible fluids evolving in the domain $\Omega(t)$ for time $t\geq 0$. The fluids are separated from each other by a moving free surface $\Gamma(t)$, this surface divides $\Omega$ into two time-dependent, disjoint, open subsets $\Omega_{\pm}(t)$ so that $\Omega= \Omega^{+}(t) \sqcup \Omega^{-}(t)\sqcup \Sigma(t) $ and $\Gamma(t)=\bar{\Omega}^{+}(t) \cap \bar{\Omega}^{+}(t)$. Their outer boundaries $\Gamma^{+}=\{y_3=1\}$ and  $\Gamma^{-}= \{y_3=-1\}$ are fixed. The fluid occupying $\Omega^{+}(t)$ is called the upper fluid and the second fluid, which occupies  $\Omega^{-}(t)$ is called the lower fluid. The two fluids  are sufficient  smooth to satisfy the pair of  incompressible MHD equations:
\begin{equation}
\begin{cases}\label{1.1}
   \partial_t u^{\pm}+u^{\pm}\cdot \nabla u^{\pm}+\nabla p^{\pm}=B^{\pm}\cdot \nabla B^{\pm},   \\
   \partial_t B^{\pm}+u^{\pm}\cdot \nabla B^{\pm}=B^{\pm}\cdot \nabla u^{\pm},\\
   \mathrm{div}u^{\pm}=0,~~\mathrm{div}B^{\pm}=0,
\end{cases}
\end{equation}
where $u^{\pm}=(u_{1}^{\pm},u_{2}^{\pm},u_{3}^{\pm})$ is the velocity field of the two fluids, $B^{\pm}=(B_{1}^{\pm},B_{2}^{\pm},B_{3}^{\pm})$ is the magnetic field of the two fluids,  $p^{\pm}$ denotes the total pressure of the two fluids in $\Omega^{\pm}$ respectively.

We assume that the upper fluid moves initially in the horizontal direction with some constant velocity and the lower fluid moves by the same constant velocity in the opposite direction, i.e, the initial velocity field is the following form:
\begin{equation}\label{1.2}
u_{0}(y_1,y_2,y_3)=\left\{
\begin{aligned}
	&(1,0,0)&y_3\geq0,\\
	&(-1,0,0)&y_3<0,
\end{aligned}
\right.
\end{equation}
and the initial magnetic field is transverse to  initial velocity field:
\begin{equation} \label{1.3}
B_{0}(y_1,y_2,y_3)=\left\{
\begin{aligned}
	&(0,a,0)&y_3\geq0,\\
	&(0,b,0)&y_3<0,
\end{aligned}
\right.
\end{equation}
where $a,b$ are constants.

For the existence of weak solutions of \eqref{1.1} by the Rankine-Hugoniot jump relations of the hyperbolic system of equations, a standard assumption is that the pressure and the normal component of the  velocity must be continuous across the free boundary $\Gamma(t)$. We also require  the normal component of the magnetic field vanish on the free boundary $\Gamma(t)$. Therefore such piecewise smooth solution should satisfy the following 
boundary conditions on $\Gamma(t)$:
\begin{equation} \label{1.4}
p^{+}= p^{-},~~u^{+}\cdot n= u^{-}\cdot n,~~ B^{+}\cdot n=0,~~B^{-}\cdot n=0  ~~\mathrm{on}~~\Gamma(t),
\end{equation}
where $n$ is the  normal vector to $\Gamma(t)$. We suppose that the fluid is confined between two rigid plane $\Gamma^+:=\{y_3=+1\}$ and $\Gamma^-:=\{y_3=-1\}$, then we impose the boundary condition at the fixed  upper and lower boundaries $\Gamma^{\pm}= \mathbb{T}^{2} \times \{y_3=\pm 1\}$:
\begin{equation} \label{1.5}
u^{\pm}_{3}(t,y_1,y_2,\pm1)= 0, \quad B^{\pm}_{3}(t,y_1,y_2,\pm1)= 0.
\end{equation}

\subsection{Lagrangian reformulation}
\quad  Our analysis in this paper relies on the reformulation of the problem under consideration in Lagrangian coordinates. We define the fixed Lagrangian domain $\Omega^{-}= \mathbb{T}^{2} \times (-1,0)$ and $\Omega^{+}= \mathbb{T}^{2} \times (0,1)$ and assume that  there is a diffeomorphism $\eta_0:\Omega^\pm\to\Omega^\pm(0)$ such that
\begin{equation}\label{1.6}
\Gamma(0)=\eta_0(\Gamma)~\mathrm{and}~\Gamma^\pm=\eta_0(\Gamma^\pm),
\end{equation}
where $\Gamma:=\{x_3=0\}$. Define the flow map $\eta^{\pm}(t,x)\in\Omega_\pm(t)$ such that for $x\in\Omega_\pm$,
\begin{equation}\label{1.7}
\begin{cases}
  \partial_t\eta^{\pm}(t,x)=u^{\pm}(\eta(t,x),x), ~t>0\\
  \eta^{\pm}(x,0)=\eta^{\pm}_0(x),
\end{cases}
\end{equation}

We denote the Euler coordinates by $(y,t)$ with $y=\eta(x,t)$, whereas we set Lagrangian coordinates by the fixed $(x,t)\in \Omega\times \mathbb{R}^{+}$. Assume that $\eta(t,\cdot)$ is invertible and define the Lagrangian unknowns in $\Omega_\pm$ as follows:
\begin{equation}\label{1.8}
(v^{\pm}, q^{\pm},H^{\pm})(t,x):=(u^{\pm}, p^{\pm}, B^{\pm})(t,\eta(t,x)), ~i,j=1,2,3.
\end{equation}
Throughout the rest paper, an equation on $\Omega$ means that the equation holds in both $\Omega^{+}$ and $\Omega^{-}$. Since the upper and lower fluids may slip across one from another,
we introduce the slip map $\sigma: \mathbb{T}^{1}\rightarrow \mathbb{T}^{1}$:
\begin{equation*}
	(\eta^{-})^{-1} \circ \eta^{+}= (\sigma(t,x_{1},x_{2}),0).
\end{equation*}
The slip map $\sigma$ gives the particle in the lower fluid that is in contact with the particle of the upper fluid on the contact surface.

We introduce the following operator notation
\begin{equation}\label{1.9}
\mathcal{A}=(\nabla\eta)^{-1},~J=\mathrm{det}(\nabla\eta),~(\nabla_\eta)_i=\mathcal{A}^{j}_{i}\partial_j,
~\mathrm{div}_\eta=\nabla_\eta\cdot,~\Delta_\eta=\nabla_\eta\cdot\nabla_\eta.
\end{equation}
Then, utilizing the chain rule and Einstein’s summation convention for repeated indexes, \eqref{1.1} can be rewritten by the following system in Lagrangian variables in $\Omega$ 
\begin{equation}\label{1.10}
\begin{cases}
  \partial_t\eta=v, &~~\mathrm{in}~~\Omega \\
\partial_t v+\nabla_{\eta}q=H\cdot \nabla_{\eta}H, &~~\mathrm{in}~~\Omega  \\
   \partial_t H=H\cdot \nabla_{\eta} v, &~~\mathrm{in}~~\Omega\\
   \mathrm{div}_{\eta}v=0, ~\mathrm{div}_{\eta}H=0 &~~\mathrm{in}~~\Omega.
\end{cases}
\end{equation}

The jump condition in Larangian coordinates  must be rewritten in terms of the slip map $\sigma$:
\begin{equation}\label{1.11}
	\begin{cases}
		\left\langle   v^{+} - v^{-}  \circ \sigma, n\right\rangle=0, &~~\mathrm{on}~~\Gamma\\
		 p^{+} = p^{-}  \circ \sigma, &~~\mathrm{on}~~\Gamma\\
		\left\langle   H^{+}, n\right\rangle=0,\left\langle   H^{-}  \circ \sigma, n\right\rangle=0 &~~\mathrm{on}~~\Gamma.
	\end{cases}
\end{equation}
where $ n=J \mathcal{A}^{T} e_{3} $ is the out normal vector of the free surface $\Gamma(t)$ which is defined by 
	\begin{equation}\label{normal out vec}
		\begin{aligned}
			n:=\partial_{1}\eta^{+}\times\partial_{2}\eta^{+}=&(\partial_{1}\eta^{+}_2\partial_2\eta^{+}_3-\partial_{1}\eta^{+}_3\partial_2\eta^{+}_2,\\
			&\partial_{1}\eta^{+}_3\partial_2\eta^{+}_1-\partial_{1}\eta^{+}_1\partial_2\eta^{+}_3,\partial_{1}\eta^{+}_1\partial_2\eta^{+}_2-\partial_{1}\eta^{+}_2\partial_2\eta^{+}_1).
		\end{aligned} 
	\end{equation}
with $e_{3}=(0,0,1)$ is  the outward unit normal to $\Gamma$.

 Using the identity $\partial_{t} \mathcal{A}^{k}_{i}= -\mathcal{A}^{s}_{i} \partial_{s} v^{r} \mathcal{A}^{k}_{r}$ and the third equation in \eqref{1.10}, we deduce that
\begin{equation}\label{1.12}
\begin{aligned}
&\partial_t (\mathcal{A}^{k}_{i} H_{k})=\partial_t \mathcal{A}^{k}_{i} H_{k}+ \mathcal{A}^{k}_{i} \partial_t H_{k} \\
&= -\mathcal{A}^{s}_{i} \partial_{s} v^{r} \mathcal{A}^{k}_{r} H_{k}+ \mathcal{A}^{k}_{i}  H^{m}\mathcal{A}^{n}_{m} \partial_{n}v_{k}=0,
\end{aligned}
\end{equation}
and hence
\begin{equation} \label{1.13}
H=H_{0} \cdot \nabla \eta.
\end{equation}

Finally the equations \eqref{1.1} and boundary conditions \eqref{1.4}, \eqref{1.5}  can be reformulated as:
\begin{equation}\label{1.15}
\begin{cases}
  \partial_t\eta=v, &~~\mathrm{in}~~\Omega \\
  \partial^{2}_t \eta+\nabla_{\eta}q=(H_{0} \cdot \nabla \eta) \cdot \nabla_{\eta}(H_{0} \cdot \nabla \eta), &~~\mathrm{in}~~\Omega  \\
   \mathrm{div}_{\eta}v=0, ~\mathrm{div}_{\eta}H=0 &~~\mathrm{in}~~\Omega,\\
    \left\langle   v^{+} - v^{-}  \circ \sigma, n\right\rangle=0, &~~\mathrm{on}~~\Gamma\\
    p^{+} = p^{-}  \circ \sigma, &~~\mathrm{on}~~\Gamma\\
    \left\langle   H^{+}, n\right\rangle=0,\left\langle   H^{-}  \circ \sigma, n\right\rangle=0 &~~\mathrm{on}~~\Gamma\\
    v_{3}=0,\quad H_{3}=0,&~~\mathrm{on}~~\Gamma^{\pm}\\
    (\eta, v, H)|_{t=0}=  (\eta_{0}, v_{0}, H_{0}).
\end{cases}
\end{equation}
\subsection{History result}
\par \quad The stability problems of an interface between two fluids in a relative motion have attracted a wide interest of researchers of various fields. This type of instability is well known as the Kelvin-Helmholtz instability which was first studied
by Hermann von Helmholtz in 1868 \cite{Helmholtz} and by William Thomson (Lord Kelvin) in
1871 \cite{Kelvin}. The Kelvin-Helmholtz (K-H) instability is important in understanding a variety of space and astrophysical phenomena involving sheared plasma flow such as the stability of the
interface between the solar wind and the magnetosphere (\cite{Dungey},\cite{Parker}), interaction between adjacent streams of different velocities in the solar wind \cite{Sturrock}  and the dynamic structure of cometary tails \cite{Ershkovich}.

Early investigations of the Kelvin-Helmholtz instability
were concerned with the magnetohydrodynamic (MHD)
stability of the tangential velocity discontinuity (velocity
shear layer of zero thickness) in an incompressible plasma.
The linear stability results are given by Chandrasekhar
\cite{Chandrasekhar}, for example. The results for two simple cases may be
noted. If the plasma flow is perpendicular to the magnetic
field (transverse case), the magnetic field has no effect on the
instability and the flow is unstable for all velocity jumps, just
as in fluid dynamics.  In our paper we shall prove this nonlinear case exhibit the Chandrasekhar' linear counterparts. If the plasma flow is parallel to the
magnetic field (parallel case), then the mode is completely
stabilized unless the total velocity jump exceeds twice the
Alfven speed.
This Kelvin-Helmholtz instability configuration is also  known in literature as the ‘vortex sheet’, as its vorticity distribution is described by a $\delta$-function supported by a discontinuity in the velocity field at the sheet location.

In the slightly refined version of
Axford \cite{Axford} and Sen \cite{Sen}, considering different magnitudes and directions of the
magnetic field on the two sides of the vortex sheet, but still assuming a discontinuity
between the two incompressible plasmas, the square of the growth rate reads
\begin{equation}\label{Sen}
	\gamma^{2}= \frac{n_{1} n_2}{(n_1+n_2)^2} [k \cdot (U_2-U_1)]^2- \frac{1}{4\pi (n_1+n_2) m_i} [(k\cdot B_1)^{2}+(k\cdot B_2)^{2}],
\end{equation}
where $m_i$ is the ion mass, $k=(0,k_y,k_z)$ is the wave vector,$ n_i$, $U_i$ and $B_i$ are
the equilibrium number density, velocity and magnetic field, on the left ($i=1$) and
on the right ($i=2$) of the vortex sheet. This equation shows how the driver for
the instability is provided by the velocity jump along the wave vector direction
and how the magnetic tension, generated by the distortion of the magnetic field
component parallel to k, stabilizes the mode. In the  case of the stabilizing effect of magnetic field over the destabilizing effect of velocity shear, this condition \eqref{Sen} simplifies to the well-known Syrovaskij condition:
\begin{equation}
|[u]|^{2}\leq 2(|h^{+}|^{2}+ |h^{-}|^{2})
\end{equation}
and
\begin{equation}
	|[u]\times h^{+}|^{2}+ |[u]\times h^{+}|^{2} \leq 2(|h^{+}\times h^{-}|^{2})
\end{equation}
Under the Syrovatskij stability condition, Morando, Trakhinin, and Trebeschi \cite{Morando} proved an a priori estimate with a loss of three derivatives for the linearized system. Trakhinin \cite{Trakhinin} proved an a priori estimate without loss of derivative from data for the linearized system with variable coefficients under a strong stability condition
\begin{equation}
	\max(|[u]\times h^{+}|, |[u]\times h^{-}|) \leq |h^{+}\times h^{-}|.
\end{equation}
In a recent work \cite{Coulombel}, Coulombel, Morando, Secchi, and Trebeschi proved an
a priori estimate without loss of derivatives for the nonlinear current-vortex sheet
problem under the strong stability condition (1.11). Nonlinear stability of the incompressible current-vortex sheet problem was proved by Sun-Wang-Zhang under Syrovatskij stability condition \cite{Sun}.  

On the other hand,  in the  case of  the destabilizing effect of velocity shear over the stabilizing effect of magnetic field, we prove  linear and nonlinear ill-posedness of the  Kelvin-Helmholtz problem for incompressible MHD fluids \cite{Xie} under the condition violating the Syrovatskij stability condition in the framwork of Ebin' work \cite{E2}. In the pages 511-512 of S. Chandrasekhar's book, he studied the effect of a magnetic field transverse to the direction of streaming and found the the development of the Kelvin-Helmholtz instability, in the direction of the streaming, is uninfluenced by the presence of the magnetic field in the transverse direction development in the linear case. Now we are going to prove the  nonlinear MHD Kelvin-Helmholtz problem exhibit the same ill-posedness as his linearized counterparts \cite{Chandrasekhar} .

\subsection{Main result}
\par \quad \quad  This paper is devoted to proving Kelvin-Helmholtz instability of the system under the destabilizing effect of velocity shear with traversal magnetic field.
\begin{theo}\label{theorem}
Suppose that the initial discontinuous velocity field and  the initial discontinuous magnetic field  satisfies  \eqref{2.2}, \eqref{2.3}. Let the initial domain to be $\Omega_{0}= \mathbb{T}^{2} \times (-1,0)\cup(0,1)$.  Then the Kelvin-Helmholtz problem of \eqref{1.15}  is  linear and nonlinear ill-posedness: for any small $\delta>0$, there exists a family of classical solutions $(\eta_{n}(t), \partial_{t} \eta_{n}(t))$ to \eqref{1.15}  such that 
\begin{equation}
\| 	\eta_{n}(0)-\xi(0) \|_{C^{\infty}(\Omega)}+  \| \partial_{t} \eta_{n}(0)-\partial_{t}\xi(0) \|_{C^{\infty}(\Omega)}\leq  \delta,
\end{equation}
but for $t>0$, we have 
\begin{equation}
	\| 	\eta_{n}(t)-\xi(t) \|_{H^{2}(\Omega)}+  \| \partial_{t} \eta_{n}(t)-\partial_{t}\xi(t) \|_{H^{2}(\Omega)} \rightarrow  \infty,
\end{equation}
as $n \rightarrow \infty$.
\end{theo}
\begin{rema}
In the linear case, the growth of the solution of the linear system is like $e^{kt}$ as time evolves, here $k$ is the tangential frequency. Thus, the solution is very unstable for the high frequency $k$. On the other hand, we show the nonlinear ill-posedness in Hadamard sense.
\end{rema}
	
\begin{rema}
This theorem expresses strong linear  and nonlinear ill-posedness of the Kelvin-Helmholtz problem for the  MHD fluids in the Sobolev framework. It is a consequence of an instability process, which holds at
high tangential frequencies. We show that some perturbations with high  tangential
frequency grow in the linear regime like  $e^{kt}$, we also prove that  it is very unlikely that this destabilization phenomenon can be cancelled by nonlinear interactions.
	\end{rema}

\subsection{Main idea}
\quad  \quad  In fluid dynamics, tangential discontinuities are always unstable with respect to infinitesimal perturbations, and so are rapidly broadened into turbulent regions. A magnetic field, however, has a stabilizing effect on the motion of a conducting fluid and in fact that a perturbation involving fluid displacement transverse to the field leads to a stretching of the lines of magnetic force frozen in it and therefore to the appearance of the force which tend to restore the unperturbed flow. However, in this paper we showed that a weak magnetic field are not able to restore the unperturbed flow, thus it will induce the famous Kelvin-Helmholtz instability. Our analysis relies on the reformulation of the problem both in Euler and Lagrangian coordinates.

The first part of the paper is devoted to a study of the equations obtained by linearizing the incompressible MHD equations, written in Lagrangian coordinates, around the steady-state solution.  A key idea is to get the evolution equation of the velocity, to begin with, the resulting linearized equations are
\begin{equation}\label{1.18}
\partial^{2}_t{v}^{+}+\nabla \partial_{s} q^{+}=  a^{2} \partial^{2}_{x_2} v^{+}  \quad \text{on}\quad\Omega^{+}.
\end{equation}
so we can find an equation for $v^{+}$ by computing $\partial_{s} q^{+}$. As $\partial_{s} q^{+}$ satisfy the elliptic equation with some boundary condition, we can find a formula for $\partial_{s} q^{+}$ on the boundary $\Gamma$. Since we are interested in Kelvin-Helmholtz instability, the instability behavior happen on the boundary. For this, we need restrict  linearized equation \eqref{1.18} to $\Gamma$, i. e., we let $x_3=0$. Then we substitute this formula for $\partial_s q^{ \pm}$ into linearized equation \eqref{1.18} on the boundary $\Gamma$, we obtain
\begin{equation}\label{1.19}
\partial^{2}_{t} v^{+}_{3}=\partial_{x_{1}}  \partial_{t}{v}^{+}_3
        +\partial_{x_{1}}  \partial_{t} v^{-}_3 \circ \sigma+\frac{1}{2} ( a^{2} \partial^{2}_{x_2}v_3^{+}+b^{2} \partial^{2}_{x_2}v_3^{-}\circ \sigma) \quad \text { on } \quad \Gamma.
\end{equation} 
To reflect the characteristic of  velocity jumping on the boundary (tangential discontinuities), we  revert the  equation \eqref{1.19} to Euler coordinates and  obtain an evolution equation for the velocity:
\begin{equation*}
\partial^{2}_{t} u^{+}_{3}+\partial_{y_{1}}^2 u^{+}_3=\frac{a^2+b^2}{2} \partial^{2}_{y_2} u^{+}_3   \quad \text { on } \quad \Gamma(t).
\end{equation*}  
Similarly, we also have
\begin{equation*}
\partial^{2}_{t} u^{-}_{3}+\partial_{y_{1}}^2 u^{-}_3=\frac{a^2+b^2}{2} \partial^{2}_{y_2} u^{-}_3   \quad \text { on } \quad \Gamma(t).
\end{equation*}  
Adding above two equations induces
\begin{equation}\label{2.3111}
\partial^{2}_{t} u_{3}+\partial_{y_{1}}^2 u_3=\frac{a^{2} +b^{2}}{2}\partial_{y_{2}}^2 u_3 \quad \text { on } \quad \Gamma(t).
\end{equation}  
Therefore  we can find  exponential growing  the solutions of \eqref{2.3111} which are linear combinations of the real and imaginary parts of $\exp(  k t+i k y_{1}), k$ is a positive integer. Here we note that the solutions of \eqref{2.3111} are indepedent of the variable $y_{2}$ on the boundary $\Gamma(t)$. It is shown that the solutions of \eqref{2.3111} are qualitatively more unstable for large frequencies $k$ in the direction of the streaming ($y_{1}$-direction). Therefore we can see that the development of the kelvin-Hellmholtz  instability, in the direction of the streaming ($y_{1}$-direction), is influenced by the presence of the magnetic field in the transverse direction  ($y_{2}$-direction).

With linear ill-posedness established, we then prove ill-posedness of the nonlinear problem. To see this ill-posedness, by the method of Ebin's work \cite{E2}, we estimate the nonlinear term and find that the nonlinear term is second order in  solution.  In order to estimate the perturb quantity $z_{n}= \eta_{n}- \xi$, we decompose it into four part, for each part, we can construct the  functional and prove the  functional is expotential growing, therefore the norm of  the perturb quantity tend to infinity, we can see that nonlinear problem is not well-posedness.

\section{The Linearized Equations in both Lagrangian coordinates and Euler coordinates}
\quad In this section we consider  a  linearized system and construct  a growing solution. We construct such solution by using both  Lagrangian coordinates and Euler coordinates.

To begin,  we  choose the initial flow map to be the identity mapping:
	\begin{equation} \label{2.1}
		\eta_0:=\eta(0)=Id,
	\end{equation}
therefore the initial data in Euler coordinates is also an initial data in Lagrangian coordinates, i.e. $v(0)$ and $H(0)$ satisfy
	\begin{equation} \label{2.2}
		v(0)=u_0(\eta_0)=\left\{
		\begin{aligned}
			&(1,0,0)&x_3>0,\\
			&(-1,0,0)&x_3<0,
		\end{aligned}
		\right.
	\end{equation}
	and
	\begin{equation} \label{2.3}
		H (0)=B_0(\eta_0)=\left\{
		\begin{aligned}
			&(0,a,0)&x_3>0,\\
			&(0,b,0)&x_3<0.
		\end{aligned}
		\right.
	\end{equation}
Now we want to find a special solution to the system  \eqref{1.15} with these data \eqref{2.1}\eqref{2.2}\eqref{2.3}, 
to  compute $p$, applying the Lagrangian divergence $div_{\eta}= \nabla_{\eta}\cdot$ to the second equation of \eqref{1.15}, Noting  the following fact, the incompressibility condition implies that $J(\eta(t,s,y))=1$,  differentiating with respect to t, we have
	$$div_{\eta_{\pm}}(\partial_{t}{\eta}^{\pm})=0.$$
Differentiating with respect to t in the above equation, we get
$$div_{\eta^{\pm}}(\partial^{2}_{t}{\eta}^{\pm})-tr(\nabla_{\eta^{\pm}}\partial_{t}{\eta}^{\pm})^2=0$$
 we get an  elliptic equation:
	\begin{equation}\label{2.4}
		\Delta_{\eta} q= tr((\nabla_{\eta}(H_{0}\cdot \nabla \eta))^2)-tr((\nabla_{\eta} \dot{\eta})^2)  ~in~\Omega.
	\end{equation}
For such initial data $\eta_{0}$, $v_{0}$ and $H_{0}$, we have $\Delta q=0$, so $q=0$ is the desired solution. Thus the system  \eqref{1.15} become $\partial^{2}_{t} \eta = 0$, and
	\begin{equation} \label{2.5}
		\eta(t,x_{1},x_2,x_3)=\left\{
		\begin{aligned}
			&(x_1+t,x_2,x_3)&x_3>0,\\
			&(x_1-t,x_2,x_3)&x_3<0,
		\end{aligned}
		\right.
	\end{equation}
	is the solution. We denote this solution as $\xi$. Now we consider a family smooth solutions  $\eta(t,s)$ to the system \eqref{1.15} and assume that $\eta(t,s=0)=\xi(t)$. We define $v(t)=\left.\partial_{s} \eta(t, s)\right|_{s=0}$,  $u(t)=v(t) \circ \eta(t)^{-1}$, and $q(t, s)=$ $p(t, s) \circ \eta(t, s)$. 
	Applying $\partial_{s}$ to the second  equation in \eqref{1.15}  and setting $s$ equal to zero, we get:
	\begin{equation} \label{2.6}
		\begin{aligned}
			\partial^{2}_t v&-\nabla_{\xi} v\cdot  \nabla_{\xi} q+\nabla_{\xi} \partial_{s} q=(H_{0} \cdot \nabla v) \cdot \nabla_{\xi}(H_{0} \cdot \nabla \xi)\\
			&-(H_{0} \cdot \nabla \xi) \cdot \nabla_{\xi} v\cdot  \nabla_{\xi}(H_{0} \cdot \nabla \xi)
			+ (H_{0} \cdot \nabla \xi) \cdot \nabla_{\xi}(H_{0} \cdot \nabla v)\quad \text{on}\quad\Omega.
		\end{aligned}
	\end{equation}
Since $\xi$ is just a horizontal translation by $t$, thus we have $\nabla_{\xi}= \nabla$.
We linearized this system around this $\xi$ to give
\begin{equation} \label{2.7}
-\nabla_{\xi} v\cdot  \nabla_{\xi} q= - \nabla (v_1,v_2)\cdot \nabla 0 =0\quad \text{on}\quad\Omega.
\end{equation}
\begin{equation} \label{2.8}
\begin{aligned}
&(H_{0} \cdot \nabla v) \cdot \nabla_{\xi}(H_{0} \cdot \nabla \xi)\\
=&\left\{
		\begin{aligned}
			&(a\partial_{x_2}(v^{+}_{1},v^{+}_{2},v^{+}_{3})) \cdot \nabla (a\partial_{x_2}(x_1+t,x_2,x_3))=0\quad &\text{on}\quad\Omega^{+},\\
			&(b\partial_{x_2}(v^{-}_{1},v^{-}_{2},v^{-}_{3})) \cdot \nabla (b\partial_{x_2}(x_1-t,x_2,x_3))=0\quad &\text{on}\quad\Omega^{-},
		\end{aligned}
		\right.
  \end{aligned}
\end{equation}
\begin{equation} \label{2.9}
\begin{aligned}
&(H_{0} \cdot \nabla \xi) \cdot \nabla_{\xi^{+}} v^{+}\cdot  \nabla_{\xi^{+}}(H_{0} \cdot \nabla \xi)\\
=&\left\{
		\begin{aligned}
			&(a\partial_{x_2}(x_1+t,x_2,x_3)) \cdot \nabla (v^{+}_1,v^{+}_2,v^{+}_{3}) \cdot \nabla (a\partial_{x_2} (x_1+t,x_2,x_3))=0& \text{on}\quad\Omega^{+},\\
			&(b\partial_{x_2}((x_1-t,x_2,x_3)) \cdot \nabla (v^{-}_1,v^{-}_2,v^{-}_{3}) \cdot \nabla (b\partial_{x_2} ((x_1-t,x_2,x_3))=0& \text{on}\quad\Omega^{-},
		\end{aligned}
		\right.
\end{aligned}
\end{equation}
and
\begin{equation} \label{2.10}
\begin{aligned}
&(H_{0} \cdot \nabla \xi) \cdot \nabla_{\xi}(H_{0} \cdot \nabla v)\\
=&\left\{
		\begin{aligned}
  &(a\partial_{x_2}(x_1+t,x_2,x_3)) \cdot \nabla (a\partial_{x_2} (v^{+}_1,v^{+}_2,v^{+}_{3})) =a^{2} \partial^{2}_{x_2} v^{+} \quad&\text{on}\quad\Omega^{+},\\
  &(a\partial_{x_2}(x_1-t,x_2,x_3)) \cdot \nabla (b\partial_{x_2} (v^{-}_1,v^{-}_2,v^{-}_{3})) =b^{2} \partial^{2}_{x_2} v^{-}\quad &\text{on}\quad\Omega^{-}.
\end{aligned}
		\right.
\end{aligned}
\end{equation}
Therefore, substituting \eqref{2.7}-\eqref{2.10} in \eqref{2.6} gives the following  linearized equations:
\begin{equation} \label{2.11}
	\partial^{2}_t{v}^{+}+\nabla \partial_{s} q^{+}=  a^{2} \partial^{2}_{x_2} v^{+}  \quad \text{on}\quad\Omega^{+}.
\end{equation}
	and
\begin{equation} \label{2.12}
	\partial^{2}_t{v}^{-}+\nabla\partial_{s} q^{-}=  b^{2} \partial^{2}_{x_2} v^{-} \quad \text{on}\quad\Omega^{-}.
\end{equation}
From this, we can see that in order to obtain an linearized equations for $v$, the remaining thing is to compute $\partial_{s} q$. Before the linearization, we know $q$ satisfy the elliptic equation \eqref{2.4}. As for the behavior at the boundary, we know that 
	\begin{equation}\label{2.13}
		q^{+}=q^{-} \circ \sigma  \quad \text {on } \quad \Gamma
	\end{equation}
It is also need to add another two conditions involving the normal derivatives of $q$, on the boundary $\Gamma^{\pm}$, since $\eta^{\pm}_{3}=\pm1, $ it implies $\partial_{t}^2\eta^{\pm}_{3}(t,y_1,\pm1)=0$, where $\eta_{3}$ is the third  component of the vector $\eta=(\eta_1,\eta_{2},\eta_{3})$, we compute
 \begin{equation}\label{2.14}
		\left\langle\nabla_{\eta} q,(0,0,1)\right\rangle= 	\left\langle(H_{0} \cdot \nabla \eta) \cdot \nabla_{\eta}(H_{0} \cdot \nabla \eta),(0,0,1)\right\rangle \text { on } \Gamma^{\pm}.
	\end{equation}
On the boundary $\Gamma$, we take the difference of the “+” and “-” momentum  equations in \eqref{1.15} to obtain the following boundary condition
	\begin{equation}\label{2.15}
		\begin{aligned}
			&\left\langle\nabla_{\eta^{+}} q^{+}-\nabla_{\eta^{-}} q^{-} \circ \sigma, n\right\rangle  =-\left\langle\partial^{2}_{t}{\eta}_{+}-\partial^{2}_{t}{\eta}_{-} \circ \sigma, n\right\rangle\\
			&+\left\langle(H_{0} \cdot \nabla \eta^{+}) \cdot \nabla_{\eta^{+}}(H_{0} \cdot \nabla \eta^{+})-  (H_{0} \cdot \nabla \eta^{-}) \cdot \nabla_{\eta^{-}}(H_{0} \cdot \nabla \eta^{-}), n \right\rangle\\
			&=: M \quad \text {on } \quad \Gamma
		\end{aligned}
	\end{equation}
Collecting the elliptic equation \eqref{2.4} and the boundary conditions \eqref{2.13}-\eqref{2.15} gives the coupled system for  the pressure 
\begin{equation}
\begin{cases}\label{2.16}
   \Delta_{\eta} q=  tr((\nabla_{\eta}(H_{0}\cdot \nabla \eta))^2)-tr((\nabla_{\eta} \dot{\eta})^2)  & \text{in} \quad\Omega^{\pm}, \\
   \left\langle\nabla_{\eta} q,(0,1)\right\rangle= 	\left\langle(H_{0} \cdot \nabla \eta) \cdot \nabla_{\eta}(H_{0} \cdot \nabla \eta),(0,0,1)\right\rangle &\text {on}\quad  \Gamma^{\pm},\\
   q^{+}=q^{-} \circ \sigma  \quad  &\text {on } \quad \Gamma, \\
   \left\langle\nabla_{\eta^{+}} q^{+}-\nabla_{\eta^{-}} q^{-} \circ \sigma, n^{+} \right\rangle=M,
    &\text {on } \quad \Gamma.
\end{cases}
\end{equation}  
  
  Applying $\left.\partial_{s}\right|_{s=0}$ to above coupled system and using the fact that $q=0$ and  $(\sigma, 0)=\xi_{-}(t)^{-1} \circ \xi_{+}(t)(x_{1},x_{2}, 0)=(x_{1}+2 t,x_{2}, 0)$ when $s=0$, we get:
\begin{equation}
\begin{cases}\label{2.17}
  \Delta \partial_{s} q= 0  & \text{in} \quad\Omega^{\pm}, \\
   \partial_{x_{3}}  \partial_{s} q=0  &\text {on}\quad  \Gamma^{\pm},\\
  \partial_{s}  q^{+}(t, x_{1},x_{2},0)= \partial_{s}  q^{-}(t, x_{1}+2t,x_{2},0) \quad  &\text {on } \quad \Gamma, \\
   \partial_{x_{3}}  \partial_{s}  q^{+}-\partial_{x_{3}}  \partial_{s} q^{-} =\partial_{s} M \quad &\text {on } \quad \Gamma.
\end{cases}
\end{equation}  

To solve \eqref{2.17} we must compute $\partial_s M$. Before computation, we need rewrite $M$ by using the boundary condition  $[v]\cdot n=0$, $n=\partial_{x_{1}}\eta^{+}\times \partial_{x_{2}}\eta^{+}$ on $\Gamma$, i.e.
\begin{equation}\label{2.18}
		\left\langle   \partial_{t}\eta^{+} - \partial_{t}\eta^{-}  \circ \sigma, n \right\rangle=0\quad \text {on } \quad \Gamma.
\end{equation}
	Differentiating \eqref{2.18} in $t$, and re-arranging terms, we get:
 \begin{equation}\label{2.19}
	\begin{aligned}
		\left\langle   \partial^{2}_{t}\eta^{+} - \partial^{2}_{t}\eta^{-}  \circ \sigma, n\right\rangle=&\left\langle 2 \partial_{x_1} \partial_t{\eta}^{-} \circ \sigma, n\right\rangle \\
		&-\left\langle   \partial_{t}\eta^{+}- \partial_{t}\eta^{-}  \circ \sigma,   \partial_{t} n\right\rangle\quad \text {on } \quad \Gamma.
	\end{aligned}
\end{equation}
	Therefore the form of $M$ can also  be written in another form:
	\begin{equation}\label{2.20}
	\begin{aligned}
		M =&\left\langle(H_{0} \cdot \nabla \eta^{+}) \cdot \nabla_{\eta^{+}}(H_{0} \cdot \nabla \eta^{+})- [(H_{0} \cdot \nabla \eta^{-}) \cdot \nabla_{\eta^{-}}(H_{0} \cdot \nabla \eta^{-})] \circ \sigma , n \right\rangle \\
&+\left\langle   \partial_{t}\eta^{+}- \partial_{t}\eta^{-}  \circ \sigma, n\right\rangle-\left\langle 2\partial_{x_1} \partial_t{\eta}^{-} \circ \sigma, n\right\rangle \quad \text {on } \quad \Gamma.
	\end{aligned}
\end{equation}
	
Applying $\left.\partial_s\right|_{s=0}$ to \eqref{2.20}, and using $\partial_{t}\eta^{\pm}(t, s=0)=\partial_{t}\xi^{ \pm}(t)= \pm(1,0,0)$ and $\left.\partial_t \sigma\right|_{s=0}=2$. Noting the following fact,
\begin{equation}\label{2.21}
	\begin{aligned}
n|_{s=0}=&(\partial_{x_1}\xi^{+}_2\partial_{x_2}\xi^{+}_3-\partial_{x_1}\xi^{+}_3\partial_{x_2}\xi^{+}_2,\partial_{x_1}\xi^{+}_3\partial_{x_2}\xi^{+}_1-\partial_{x_1}\xi^{+}_1\partial_{x_2}\xi^{+}_3,\\
&\partial_{x_1}\xi^{+}_1\partial_{x_2}\xi^{+}_2-\partial_{x_1}\xi^{+}_2\partial_{x_2}\xi^{+}_1)=(0,0,1),
\end{aligned}
\end{equation}
and 
\begin{equation}\label{2.22}
\begin{aligned}
&\quad\partial_{s}n|_{s=0}\\&=(-\partial_{x_1}v_3^{+},-\partial_{x_2}v_3^{+},\partial_{x_1}v_1^{+}+\partial_{x_2}v_2^{+})\\&=(-\partial_{x_1}v_3^{+},-\partial_{x_2}v_3^{+},-\partial_{x_3}v_3^{+}),
\end{aligned}
\end{equation}

By taking use of \eqref{2.8}-\eqref{2.10} and \eqref{2.21}-\eqref{2.22}  , we obtain
\begin{equation*}
	\begin{aligned}
	&\partial_s \left\langle(H_{0} \cdot \nabla \eta^{+}) \cdot \nabla_{\eta^{+}}(H_{0} \cdot \nabla \eta^{+}), n\right\rangle |_{s=0}\\
=&\left\langle(H_{0} \cdot \nabla v^{+}) \cdot \nabla_{\xi}(H_{0} \cdot \nabla \xi^{+}) , n|_{s=0}\right\rangle\\
		&- \left\langle(H_{0} \cdot \nabla \xi^{+}) \cdot \nabla_{\xi^{+}} v^{+}\cdot  \nabla_{\xi^{+}}(H_{0} \cdot \nabla \xi^{+}) ,n|_{s=0}\right\rangle \\
		&+\left\langle(H_{0} \cdot \nabla \xi^{+}) \cdot \nabla_{\xi^{+}}(H_{0} \cdot \nabla v^{+}) , n|_{s=0}\right\rangle\\
		&+\left\langle(H_{0} \cdot \nabla \xi^{+}) \cdot \nabla_{\xi^{+}}(H_{0} \cdot \nabla \xi^{+}) , \partial_{s}n|_{s=0} \right\rangle \\
=&\left\langle(a \partial_{x_2} v^{+}) \cdot \nabla(a \partial_{x_2}(x_1+t,x_2)) ,(0,0,1) \right\rangle\\
		&- \left\langle(a \partial_{x_2}  (x_1+t,x_2,x_3)) \cdot \nabla v^{+}\cdot  \nabla(a \partial_{x_2}(x_1+t,x_2,x_3)), (0,0,1)\right\rangle \\
		&+\left\langle(a \partial_{x_2}  (x_1+t,x_2,x_3)) \cdot \nabla(a \partial_{x_2}(v_1,v_2,v_3)) , (0,0,1)\right\rangle\\
		&+\left\langle(a \partial_{x_2}  (x_1+t,x_2,x_3)) \cdot \nabla(a \partial_{x_2}(x_1+t,x_2,x_3)), (-\partial_{x_1}v_3^{+},-\partial_{x_2}v_3^{+},-\partial_{x_3}v_3^{+})\right\rangle\\
=&a^{2} \partial^{2}_{x_2}v_3^{+}.	
	\end{aligned}
\end{equation*}
Similarly, we derive that
\begin{equation*}
\partial_s \left\langle [(H_{0} \cdot \nabla \eta^{-}) \cdot \nabla_{\eta^{-}}(H_{0} \cdot \nabla \eta^{-})] \circ \sigma, n \right\rangle |_{s=0}
=b^{2} \partial^{2}_{x_2}v_3^{-} \circ \sigma.	
\end{equation*}

	For the remaining two term in \eqref{2.20}, noting the fact
	\begin{equation} \label{2.26}
		\begin{aligned}
			&\quad\partial_{t}n |_{s=0}=0,
		\end{aligned}
	\end{equation}
	and 
	\begin{equation}\label{2.27}
		\begin{aligned}
			&\quad\partial_{t}\partial_{s} n|_{s=0}\\&=(-\partial_{t}\partial_{x_1}v_3^{+},-\partial_{t}\partial_{x_2}v_3^{+},\partial_{t}\partial_{x_1}v_1^{+}+\partial_{t}\partial_{x_2}v_2^{+})\\&=(-\partial_{t}\partial_{x_1}v_3^{+},-\partial_{t}\partial_{x_2}v_3^{+},-\partial_{t}\partial_{x_3}v_3^{+}),
		\end{aligned}
	\end{equation}
	In according with \eqref{2.26} and \eqref{2.27},  we have 
	\begin{equation}\label{2.28}
		\begin{aligned}
			\partial_{s}\left\langle   \partial_{t}\eta^{+}- \partial_{t}\eta^{-}  \circ \sigma,  \partial_{t}n\right\rangle\bigg{|}_{s=0}&=0+\left\langle   \partial_{t}\xi^{+}- \partial_{t}\xi^{-}  \circ \sigma,  \partial_{t}\partial_{s}n\right\rangle\bigg{|}_{s=0}\\
			&=-2\partial_{t}\partial_{x_1}v_3^{+},
		\end{aligned}
	\end{equation}
	and 
	\begin{equation}\label{2.29}
		\partial_{s}\left\langle 2 \partial_{x_1} \partial_t{\eta}^{-} \circ \sigma,n\right\rangle\bigg{|}_{s=0}=2\partial_{t}\partial_{x_1}v_3^{-}\circ \sigma,
	\end{equation}

Therefore we have
\begin{equation*}
	\begin{aligned}
		\partial_s M=a^{2} \partial^{2}_{x_2}v_3^{+}-b^{2} \partial^{2}_{x_2}v_3^{-}\circ \sigma  -2\left(\partial_{x_{1}}  \partial_{t} v^{-}_3 \circ \sigma+\partial_{x_{1}}  \partial_{t}{v}^{+}_3\right),
	\end{aligned}
\end{equation*}
Thus the forth equation in \eqref{2.17} becomes:
\begin{equation}\label{2.21b}
	\partial_{x_{3}}  \partial_{s}  q^{+}-\partial_{x_{3}}  \partial_{s} q^{-} =a^{2} \partial^{2}_{x_2}v_3^{+}-b^{2} \partial^{2}_{x_2}v_3^{-}\circ \sigma  -2\left(\partial_{x_{1}}  \partial_{t} v^{-}_3 \circ \sigma+\partial_{x_{1}}  \partial_{t}{v}^{+}_3\right).
\end{equation}
Collecting \eqref{2.21b} and \eqref{2.17} gives the linear system for $\partial_s q$
\begin{equation}\label{2.22b}
\begin{cases}
  \Delta \partial_{s} q= 0  & \text{in} \quad\Omega^{\pm}, \\
   \partial_{x_{3}}  \partial_{s} q=0  &\text {on}\quad  \Gamma^{\pm},\\
  \partial_{s}  q^{+}(t,s, x_{1},x_{2},0)= \partial_{s}  q^{-}(t,s, x_{1}+2t,x_{2},0) \quad  &\text {on } \quad \Gamma, \\
   \partial_{x_{3}}  \partial_{s}  q^{+}-\partial_{x_{3}}  \partial_{s} q^{-} =a^{2} \partial^{2}_{x_2}v_3^{+}-b^{2} \partial^{2}_{x_2}v_3^{-}\circ \sigma \\
   -2\left(\partial_{x_{1}}  \partial_{t} v^{-}_3 \circ \sigma+\partial_{x_{1}}  \partial_{t}{v}^{+}_3\right)\quad &\text {on } \quad \Gamma.
\end{cases}
\end{equation}  

Now we can 	 find an formula for the  $\partial_s q$ using reflection across $\mathbb{T}^1$. Let $\partial_s q$ be the harmonic function on $\Omega_{+}$with boundary data:
	\begin{equation}\label{2.23b}
	\partial_{x_3} \partial_s q^{+}=\left\{\begin{array}{ccc}
		0 & \text { at } & x_3=1 \\
		-\left(\partial_{x_{1}}  \partial_{t} v^{-}_3 \circ \sigma+\partial_{x_{1}}  \partial_{t}{v}^{+}_3\right)\\
        +\frac{1}{2} ( a^{2} \partial^{2}_{x_2}v_3^{+}+b^{2} \partial^{2}_{x_2}v_3^{-}\circ \sigma)& \text { at } & x_3=0
	\end{array}\right\}
	\end{equation} 
	Then let $\partial_s q^{-}(x_{1}+2 t,-x_2,-x_3)=\partial_s q^{+}(x_{1}, x_2,x_3)$. One can verify that   $\partial_s q^{-}(x_{1}+2 t,-x_2,-x_3)$ satisfy three boundary conditions in \eqref{2.22} . Therefore we have an formula for $\partial_s q^{ \pm}$ on the boundary $\Gamma$.  Since we are interested in Kelvin-Helmholtz instability, the instability behavior happen on the boundary. For this, we need restrict  linearized equation \eqref{2.11} to $\mathbb{T}^2 ;$ i. e., we let $x_3=0$. Then we substitute this formula for $\partial_s q^{ \pm}$ into linearized equation \eqref{2.11} on the boundary $\Gamma$, we obtain
	\begin{equation}\label{2.24b}
	\partial^{2}_{t} v^{+}_{3}=\partial_{x_{1}}  \partial_{t}{v}^{+}_3
        +\partial_{x_{1}}  \partial_{t} v^{-}_3 \circ \sigma+\frac{1}{2} ( a^{2} \partial^{2}_{x_2}v_3^{+}+b^{2} \partial^{2}_{x_2}v_3^{-}\circ \sigma) \quad \text { on } \quad \Gamma.
	\end{equation} 
In according with the boundary condition $ \left\langle   v^{+} - v^{-}  \circ \sigma, n\right\rangle=0$ in \eqref{1.15}, via linearing around the special solution $\xi$, we have
\begin{equation}\label{linear boundary}
v_3^+(t,x_1,x_2,0)=v_3^-(t,x_1+2t,x_2,0)\quad \text { on } \quad \Gamma. 
\end{equation}
Applying $\partial_{x_1}\partial_{t}$ and $\partial_{x_2}^2$ to the above equation, we have
\begin{equation}\label{derive1}
\begin{aligned}
\partial_{x_1}\partial_{t}v_3^+(t,x_1,x_2,0)=&\partial_{x_1}\partial_{t}v_3^-(t,x_1+2t,x_2,0)\\
&+2\partial_{x_1}^2v_3^-(t,x_1+2t,x_2,0)\quad \text { on } \quad \Gamma,
\end{aligned}
\end{equation}
and
\begin{equation}\label{derive2}
\partial_{x_2}^2v_3^+(t,x_1,x_2,0)=\partial_{x_2}^2v_3^-(t,x_1+2t,x_2,0)\quad \text { on } \quad \Gamma. 
\end{equation}
Then, by \eqref{derive1} and \eqref{derive2}, it holds
\begin{equation}
\begin{aligned}
\partial_{x_{1}} \partial_{t} v^{+}_{3}+\partial_{x_{1}}\partial_{t} v^{-}_{3}\circ \sigma&=2\partial_{x_{1}} \partial_{t} v^{+}_{3}-2\partial_{x_1}^2v_3^-\circ \sigma\\
&=2\partial_{x_{1}} \partial_{t} v^{+}_{3}-2\partial_{x_1}^2v_3^+ \quad on ~
\Gamma,
\end{aligned}
\end{equation}
and
\begin{equation}
\begin{aligned}
 \frac{1}{2}(a^2\partial^{2}_{x_2}v^+_2+b^2\partial^{2}_{x_2}v^-_2\circ \sigma)= \frac{1}{2}(a^2\partial^{2}_{x_2}v^+_2+b^2\partial^{2}_{x_2}v^+_2) \quad on ~\Gamma.
\end{aligned}
\end{equation}
Therefore \eqref{2.24} becomes 
	\begin{equation}\label{2.2444}
	\partial^{2}_{t} v^{+}_{3}=2\partial_{x_{1}} \partial_{t} v^{+}_{3}-2\partial_{x_1}^2v_3^+   + \frac{a^2+b^2}{2}\partial^{2}_{x_2}v^+_3\quad \text { on } \quad \Gamma.
	\end{equation} 
    
To reflect the characteristic of  velocity jumping on the boundary, we  revert the \eqref{2.2444} to Euler coordinates. To do this we define:
\begin{equation*}
	v^{+}(t, x_{1},x_{2},x_{3})=u^{+}(t,y_{1},y_{2},y_{3})=u^{+}(t, x_{1}+t,x_{2},x_{3}),
\end{equation*} 
\begin{equation*}
	v^{-}(t, x_{1}+2t,x_{2},x_{3})=u^{-}(t,y_{1},y_{2})=u^{-}(t, x_{1}+t,x_{2},x_{3}).
\end{equation*} 
Then by chain rule,  for $v^{+}$ we have 
\begin{equation}\label{2.25}
	\begin{aligned}
 &\partial_{x_{1}} \partial_{t} v^{+}_{3}=\partial_{y_{1}} \partial_{t} u^{+}_{3}+\partial^{2}_{y_{1}} u^{+}_3, \partial_{x_{1}}^{2}v^{+}_{3}=\partial^{2}_{y_{1}} u^{-}_3,\\
&\partial^{2}_{t} v^{+}_{3}=\partial^{2}_{t} u^{+}_{3}+2\partial_{y_{1}} \partial_t u^{+}_3+\partial_{y_{1}}^2 u^{+}_3.
\end{aligned}
\end{equation}
 
Then  the linearized equations \eqref{2.2444} is rewritten as follows:
\begin{equation}\label{2.29b}
	\partial^{2}_{t} u^{+}_{3}+\partial_{y_{1}}^2 u^{+}_3=\frac{a^2+b^2}{2} \partial^{2}_{y_2} u^{+}_3   \quad \text { on } \quad \Gamma(t).
\end{equation}  
Similarly, we also have
\begin{equation}\label{2.30b}
	\partial^{2}_{t} u^{-}_{3}+\partial_{y_{1}}^2 u^{-}_3=\frac{a^2+b^2}{2} \partial^{2}_{y_2} u^{-}_3   \quad \text { on } \quad \Gamma(t).
\end{equation}  
In fact, \eqref{linear boundary} implies $u_3^{+}=u_3^{-}:=u_3$ on $\Gamma(t)$.
Adding equations   \eqref{2.29b} and \eqref{2.30b} gives us
\begin{equation}\label{2.31}
	\partial^{2}_{t} u_{3}+\partial_{y_{1}}^2 u_3=\frac{a^{2} +b^{2}}{2}\partial_{y_{2}}^2 u_3 \quad \text { on } \quad \Gamma(t).
\end{equation}

Clearly the solutions of \eqref{2.31} are linear combinations of the real and imaginary parts of $\exp(  k t+i k y_{1}), k$ is a positive integer. Here we note that the solutions of \eqref{2.31} are indepedent of the variable $y_{2}$ on the boundary $\Gamma(t)$.  the solutions of \eqref{2.31} are qualitatively more unstable for large frequencies $k$ in the direction of the streaming ($y_{1}$-direction). Therefore we can see that the development of the kelvin-Hellmholtz  instability, in the direction of the streaming ($y_{1}$-direction), is influenced by the presence of the magnetic field in the transverse direction ($y_{2}$-direction). 
	
	Now we are ready to solve \eqref{2.11} and \eqref{2.12}. Firstly,  we let $U$ be the space of $R^2$-valued functions satisfying boundary \eqref{1.4} and \eqref{1.5},  we decompose $U$ into $U_0 \oplus U_h$,where $U_{0}$  is defined as : $U_{0} = \{ u: \Omega \rightarrow \mathbb{R}^{2} |  u_{3}(y_{1}, y_{2},0)=0\} $ and  the elements of $U_{h}$  are harmonic. If $u\in U_{0}$, it follows that $u_{3}=v_{3}=0$ on the boundary $\Gamma$, then
	we can see that  $\partial_{s} q$ satisfy the following system
	\begin{equation*}
	\begin{cases}
	\Delta \partial_{s} q= 0  & \text{in} \quad\Omega^{\pm}, \\
	\partial_{x_{3}}  \partial_{s} q=0  &\text {on}\quad  \Gamma^{\pm},\\
	\partial_{s}  q^{+}(t, x_{1}, x_{2},0)= \partial_{s}  q^{-}(t, x_{1}+2t,x_{2},0) \quad  &\text {on } \quad \Gamma, \\
	\partial_{x_{3}}  \partial_{s}  q^{+}-\partial_{x_{3}}  \partial_{s} q^{-}\circ \sigma =0\quad &\text {on } \quad \Gamma.
	\end{cases}
	\end{equation*}  
	From this, we find that $\partial_s q$ is zero, therefore \eqref{2.11} and \eqref{2.12}  becomes calssical wave equations:
	\begin{equation*}
	\partial^{2}_t{v}^{+}=  a^{2} \partial^{2}_{x_2} v^{+}  \quad \text{on}\quad\Omega^{+}.
	\end{equation*}
	and
	\begin{equation*} 
	\partial^{2}_t{v}^{-}=  b^{2} \partial^{2}_{x_2} v^{-} \quad \text{on}\quad\Omega^{-}.
	\end{equation*}
	We can find  that the  solution of \eqref{2.11} and \eqref{2.12} are:
	\begin{equation}\label{2.32}
	v(t)=v(0)+ t \partial_{t} v(0).
	\end{equation}  
.	

	If $u\in U_{h}$,  the elements of $U_{h}$  are harmonic, we extend $u_{3}$ on the boundary to the whole domain, i.e. we assume $u_{3}(t,y_{1},y_{2},y_{3})= e^{ k}  e^{i k y_{1}} W(y_{3})$.  To make sure that the solution satisfy the equation and boundary conditions, we deduce that $W(y_{3})$ must satisfy the system:
\begin{equation*}
	\begin{cases}
	\partial_{y_{2}}^{2}W(y_{3})-k ^{2}W(y_{3})=0  & \text{in} \quad\Omega, \\
	W(y_{3}=\pm 1)=0  &\text {on}\quad  \Gamma^{\pm},\\
	W(y_{3}=0)^{+}=W(y_{3}=0)^{-} \quad  &\text {on } \quad \Gamma,
	\end{cases}
	\end{equation*}   
Following the construction in Ebin paper \cite{E2}, we can find the solutions of this system are linear combinations of the real and imaginary parts of the functions
	\begin{equation}\label{2.33}
	W(y_{3})=\left\{
	\begin{aligned}
	&\cosh k y_{3}- \coth k \sinh ky_{3}&y_{3}\geq0,\\
	&\cosh k y_{3}+ \coth k \sinh ky_{3} &y_{3}<0,
	\end{aligned}
	\right.
	\end{equation}
Then taking use of incompressibility condition  $\partial_{y_{1}}u_{1}= -  \partial_{y_{3}}u_{3}$, we see that $u_{1}= e^{\lambda k}  e^{i k y_{1}} V(y_{3})$ where $V(y_{3})$ is the solution of the following system
\begin{equation}\label{2.33}
	V(x_{3})=\left\{
	\begin{aligned}
	&i(\sinh k y_{3}- \coth k \cosh ky_{3})&y_{3}\geq0,\\
	&i(\sinh k y_{3}+ \coth k \cosh ky_{3}) &y_{3}<0,
	\end{aligned}
	\right.
	\end{equation}
 
From above construction,  we can deduce  that if $u\in U_{h}$, then	
	\begin{equation}\label{2.35}
\partial^{2}_{t} u+\partial_{y_{1}}^2 u=\frac{a^{2} +b^{2}}{2}\partial_{y_{2}}^2 u \text { on } \Omega .
	\end{equation}
	Therefore we see  that the  solution of \eqref{2.11} and \eqref{2.12} are following form:
	\begin{equation}\label{2.36}
	v(t,x_{1}, x_{2}, x_{3})= P_{a}\left\{\begin{array}{ll}
	e^{ \pm  k t}e^{i k (x_{1}+t)}\left(W(x_{3}) ,0,V(x_{3}) \right) & x_{3} \geq 0 \\
	e^{ \pm  k t}e^{i k (x_{1}-t)}\left(W(x_{3}) ,0,V(x_{3}) \right)& x_{3}<0
	\end{array}\right\}
	\end{equation}
	where " $\mathrm{Pa}$ " means real or imaginary part. Equations \eqref{2.32} and \eqref{2.36} provide us with a full set of solutions of the linearized Helmholtz problem \eqref{2.11} and \eqref{2.12}.
	
	Since $k\rightarrow \infty$, the solutions \eqref{2.36} with  a higher frequency grow faster in time, which provides a mechanism for Kelvin-Helmholtz instability.
	
	\section{Discontinuous Dependence on Initial Data}
	To show ill-posedness of  Helmholtz problems, we will follow the basic method of [E]. Firstly, we define the operator $R_{\eta}f:=f\circ\eta$, then $R_{\eta^{-1}}f:=f\circ\eta^{-1}$, thus we define the differential operator $A_{\eta}$ by $A_{\eta}:=R_{\eta}AR_{\eta^{-1}}$.
 \begin{lemm}\label{nonlinear parameter}
		Via parameterization of $\zeta(t,x)$, we take $\eta(t,s,x):=\zeta(t,x)+sv(t,x)+o(s)$, then we have the following operator equality.
		\begin{enumerate}
		\item  $\partial_{s}\nabla_{\eta(t,s)}|_{s=0}=R_{\eta}[u\cdot\nabla,\nabla]R_{\eta^{-1}}=[u\cdot\nabla,\nabla]_{\eta}$;
			\item  $\partial_{s}\Delta_{\eta(t,s)}|_{s=0}=R_{\eta}[u\cdot\nabla,\Delta]R_{\eta^{-1}}=[u\cdot\nabla,\Delta]_{\eta}$;
				\item $\Delta^{-1}_{\eta}=(\Delta_{\eta})^{-1}$;
			\item  	$\partial_{s}\Delta^{-1}_{\eta(t,s)}|_{s=0}=([u\cdot\nabla,\Delta^{-1}]-\mathcal{H}(u\cdot\nabla)\Delta^{-1})_{\eta}$,
				where $\mathcal{H}$ projects a functon onto its harmonic parts, i.e. 
			$$\mathcal{H}: f\longmapsto \mathcal{H}f $$ with $\Delta \mathcal{H}f =0$.
			\item Let $A,B$ be some differential operator, then it holds $$(AB)_{\eta}=A_{\eta}B_{\eta}, ~(A-B)_{\eta}=A_{\eta}-B_{\eta}$$
			and $$[A,B]_{\eta}=[A_{\eta},B_{\eta}]$$
			\end{enumerate}
		
\end{lemm}
\begin{proof}
	The proof of item 1,2,4 can be refered to [E]. 
	Since 
	$$(R_{\eta}\Delta R_{\eta^{-1}})^{-1}=R_{\eta}\Delta^{-1} R_{\eta^{-1}},$$
	The item 3 is proved.
	Since $$(AB)_{\eta}=R_{\eta}AR_{\eta^{-1}}R_{\eta}BR_{\eta^{-1}}=A_{\eta}B_{\eta},$$
	$$(A-B)_{\eta}=R_{\eta}(A-B)R_{\eta^{-1}}=R_{\eta}AR_{\eta^{-1}}-R_{\eta}BR_{\eta^{-1}},$$
	$$[A,B]_{\eta}=(AB-BA)_{\eta}=A_{\eta}B_{\eta}-B_{\eta}A_{\eta}$$
	Thus item 5 is proved.
\end{proof}
We state the following lemma to stress the equivalence of the two kinds of definition of $\nabla_{\eta}$.
\begin{lemm}
	 If $\mathcal{A}$ is defined by \eqref{1.9}, we have the equality 
 $\nabla_{\eta}=\mathcal{A}\nabla=R_{\eta}\nabla R_{\eta^{-1}}$.
\end{lemm}
\begin{proof}
	Noting the fact $\eta^{-1}\circ\eta=I_d$, where $I_d$ is the identity map. Applying the differential operator $D$, thus $D\eta^{-1}\circ\eta\cdot D\eta=I_{2\times2}$, where $I_{2\times2}$ is the unit $2\times2$ matrix. Then $D\eta^{-1}\circ\eta=(D\eta)^{-1}$.
	Since
	\begin{align}
R_{\eta}\nabla^{i}R_{\eta^{-1}}f=\partial_{i}(f\circ\eta^{-1})\circ\eta=[\partial_{j}f\circ\eta^{-1}\cdot\partial_{i}\eta_{j}^{-1}]\circ\eta=(D\eta)^{-1}_{ij}\partial_{j}f
	\end{align} 
	Then $R_{\eta}\nabla R_{\eta^{-1}}f=\mathcal{A}\nabla f=\nabla_{\eta}f$.
	\end{proof}
Moreover, we have the following fact.
\begin{lemm}\label{lem:lemm1}
It holds $D_{\eta}\eta_{t}=(D\eta)^{-1}D\eta_{t}$.
\end{lemm}
\begin{proof}
	Since it holds $$D_{\eta}^{j}\eta_{t}=R_{\eta}D_{j}R_{\eta^{-1}} \eta_{t}=(\partial_{j}\eta_{t}\circ\eta^{-1}\cdot\partial_{i}(\eta^{-1})_{j})\circ\eta=\partial_{i}(\eta^{-1})_{j}\circ\eta\partial_{j}\eta_{t},$$
	which implies $D_{\eta}\eta_{t}=(D\eta)^{-1}D\eta_{t}$.
\end{proof}
\subsection{Estimates of nonlinear terms.}
\quad \quad  Next, we give the unstability of the nonlinear equations and estimate the nonlinear terms of nonlinear equations \eqref{1.15}.

	In order to show the ill-posedness of the nonlinear equations \eqref{1.15}, we need to find an equation for the perturbation  $\chi_{n}(t):=\eta_{n}(t)-\xi(t)$.
 Then, we introduce some function spaces as following,
 	$$
 	\begin{aligned}
M:=&\left \{\eta^{\pm }: \Omega^{\pm} \rightarrow R^{3}, \big |~ 
	J(\eta^{\pm})=1, \right.  (\eta^{-})^{-1} \circ \eta^{+}= (\sigma(t,x_{1}),0), \\& \text{where}~  \sigma: \mathbb{T}^{2}\rightarrow \mathbb{T}^{2}, ~\eta_{3}^{\pm}(x, \pm 1)=\pm 1, \left.\eta|_{\bar{\Omega}_{+}} \right.\ \text {is smooth }  \text {and}\\& \eta|_{\Omega_{-}} \text {extends smoothly to } \bar{\Omega}_{-}\left. \right \} 
 	\end{aligned}
 $$
 where $\eta _{2}$ means the second component of $\eta$. And its tangent space is:
 	$$
 	\begin{aligned}
T_{\eta} M:=&\left \{v^{\pm }: \Omega^{\pm } \rightarrow R^{3}, H^{\pm }: \Omega^{\pm } \rightarrow R^{3}\big| 
	{\rm div} v^{ \pm}=0, {\rm div}  H^{ \pm}=0, \right. \  v^{\pm }_{3}(x, \pm 1)=0,\\& H^{\pm }_{3}(x,\pm 1)=0,  \left\langle   v^{+} - v^{-}  \circ \sigma,\left. n\right\rangle=0 \right.\ \text{and} \\& 	\left\langle   H^{+}, n\right\rangle=0, \left\langle   H^{-}  \circ \sigma, n\right\rangle=0 ~\mathrm{on}~ \Gamma,\left.v|_{\bar{\Omega}_{+}} \right.\ \text {is smooth }  \text {and}\\& v|_{\Omega_{-}} \text {extends smoothly to } \bar{\Omega}_{-}\left. \right \} 
 	\end{aligned}
 $$
 where $v_{3}$ and $H_{3}$ mean the third component of $v$ and $H$, respectively.
	
	Unfortunately, however, since $M$ is not an affine space, when  $\eta_{n}(t)$ and $\xi(t)$ belong to $M$, but $\eta_{n}(t)-\xi(t)\notin M$. Thus, we shall consider a bigger affine space $N$ which contains $M$ as follows:
 	$$
 	\begin{aligned}
N:= \{&\eta^{\pm }: \Omega^{\pm} \rightarrow R^{3}, \big |~\eta_{3}^{\pm}(x, \pm 1)=\pm 1, \eta|_{\bar{\Omega}_{+}} \ \text {is smooth }  \text {and}\\& \eta|_{\Omega_{-}} \text {extends smoothly to } \bar{\Omega}_{-} \} 
 	\end{aligned}
 $$
the corresponding tangent space at $\zeta$ is:
	
	$$
	\begin{aligned}
		T_{\eta} N \equiv & \left\{v^{\pm }: \Omega^{\pm } \rightarrow R^{3} |v^{\pm }_{3}(x, \pm 1)=0,\right.  v|_{\Omega_{+}} \text {is smooth } \\
		& \text { and } \left.v|_{\Omega_{-}} \text {extends smoothly to } \bar{\Omega}_{-}\right\} .
	\end{aligned}
	$$
Now we see that 	$\eta_{n}(t)-\xi(t)\in T_{\eta}N $.
	
	To find a differential equation for $\eta_{n}(t)-\xi(t)$, similar to Ebin \cite{E2}, we define:
	$$
	Z(\eta_n, \partial_{t}{\eta}_n)=(H_{0} \cdot \nabla \eta_n) \cdot \nabla_{\eta_n}(H_{0} \cdot \nabla \eta_n)-\nabla_{\eta_n} q.
	$$
Therefore, the nonlinear problem can be  rewritten as follows:
	
	$$
\partial_{t}^{2}{\eta}_n=Z(\eta_n,  \partial_{t}{\eta_n}).
	$$
	
	Now let $\chi_{n}=\eta_{n}-\xi$. Then suppressing the subscript $n$, we find that $\chi(t)$ satisfies:
	\begin{equation}\label{zequation}
		\begin{aligned}
	\partial_{t}^{2}{\chi} & =Z(\eta, \partial_{t}{\eta})-Z(\xi, \partial_{t}{\xi}) \\
		& =\int_{0}^{1} D Z(\alpha(s), \partial_{t}{\alpha}(s))(\chi, \partial_{t}{\chi}) d s,
		\end{aligned}
	\end{equation}
where $\alpha(s)$ is defined as follows
	$$
\alpha(s)=\xi+s(\eta-\xi),
	$$
and $D Z(\alpha(s), \partial_{t}{\alpha}(s))(\chi, \partial_{t}{\chi})$ means the derivative of $Z$ at $(\alpha(s), \partial_{t}{\alpha}(s))$ in direction $(\chi, \partial_{t}{\chi})$. By direct calculation, \eqref{zequation} becomes:
	\begin{equation*}
		\begin{aligned}
\partial_{t}^{2}{\chi}=&D Z(\xi, \partial_{t}{\xi})(\chi, \partial_{t}{\chi})\\
&+\int_{0}^{1}(D Z(\alpha(s), \partial_{t}{\alpha}(s))(\chi, \partial_{t}{\chi})-D Z(\xi, \partial_{t}{\xi})(\chi, \partial_{t}{\chi}) d s,
\end{aligned}
\end{equation*}
then
	\begin{equation}\label{double integration}
		\begin{aligned}
	\partial_{t}^{2}{\chi}=&D Z(\xi, \partial_{t}{\xi})(\chi, \partial_{t}{\chi})\\
	&+\int_{0}^{1} \int_{0}^{s} D^{2} Z(\alpha(\sigma), \partial_{t}{\alpha}(\sigma))((\chi, \partial_{t}{\chi}),(\chi, \partial_{t}{\chi})) d \sigma ds.
	\end{aligned}.
\end{equation}

 	It is clear \eqref{double integration} is simply the linearized term of $(\chi, \partial_{t}{\chi})$ plus the double integral term. We shall consider it as a linear equation in $(\chi, \partial_{t}{\chi})$ with inhomogeneous term provided by the double integral. Before we estimate the nonlinear terms, we need to introduce four function space.  Let $M^{s}$ be the $H^{s}$-completion of our configuration $M$, we denote $M^{s}=M\cap H^{s}$, similarly we denote $N^{s}=N\cap H^{s}$ and the corresponding tangent spaces  $T_{\eta} M^{s}=T_{\eta} M\cap H^{s} $ and $T_{\eta} N^{s}=T_{\eta} N\cap H^{s} $. The first step in estimating the solution of \eqref{double integration} is to estimate $D^{2} Z$. To do this we first compute $D Z(\eta, \partial_{t}{\eta})$ for $\eta$ near $\xi$ and then estimate the second derivative. Before computing the second derivative of $Z(\eta, \partial_{t}{\eta})$, we state some vital lemmas.

	\begin{lemm}
The unique solution of the two-phase Poission equations \eqref{2.16} can be represented by the sum of solutions of two different elliptic equations \eqref{decompose1}, \eqref{decompose2}.The precise form of the solution is given by 
\begin{equation} \label{total pressure}
q=q_{1}+q_{2}=\mathcal{H}^{\partial}_{\eta}(0,M)+\Delta_{\eta}^{-1}[tr((\nabla_{\eta} \partial_{t}{\eta})^2)+tr((\nabla_{\eta}(B_{0}\cdot \nabla \eta))^2)].
\end{equation}
	\end{lemm}
	\begin{proof}
	We first decomopse the system \eqref{2.16} into the following two systems
	\begin{equation}\label{decompose1}
	 \mathbf{I}:\left\{\begin{array}{ll}
  	\Delta_{\eta} q=  0 \quad &\text{on}\quad \Omega^{\pm},\\
  		\left\langle\nabla_{\eta} q,(0,1)\right\rangle=0 \quad&\text { on }\quad\Gamma^{\pm},\\
  			q^{+}-q^{-} \circ \sigma=0  \quad &\text {on } \quad \Gamma,\\
  			\left\langle\nabla_{\eta^{+}} q^{+}-\nabla_{\eta^{-}} q^{-} \circ \sigma, n\right\rangle =M(\eta,\partial_{t}{\eta}) \quad &\text {on } \quad \Gamma,
  			\end{array}\right.
	\end{equation}
 and
		\begin{equation}\label{decompose2}
	 \mathbf{II}:\left\{\begin{array}{ll}
	\Delta_{\eta} q=   tr((\nabla_{\eta}(B_{0}\cdot \nabla \eta))^2) -tr((\nabla_{\eta} \partial_{t}{\eta})^2) \quad &\text{on}\quad \Omega^{\pm}+ ,\\
	\left\langle\nabla_{\eta} q,(0,1)\right\rangle=0 \quad&\text { on }\quad\Gamma^{\pm},\\
	q^{+}-q^{-} \circ \sigma=0  \quad &\text {on } \quad \Gamma,\\
	\left\langle\nabla_{\eta^{+}} q^{+}-\nabla_{\eta^{-}} q^{-} \circ \sigma, n\right\rangle =0 \quad &\text {on } \quad \Gamma,
	\end{array}\right.
	\end{equation}
	Indeed, the solution of \eqref{decompose1},\eqref{decompose2} can be defined up to an addictive constant which can be proved by the elliptic estimate as in \cite{L}.
 Since the extension map $\mathcal{E}_{\eta}$ and harmonic map $\mathcal{H}_{\eta}$ are defined in \cite{E1}, page 1183-1285, thus we define $\mathcal{H}^{\partial}_{\eta}$ to be $\mathcal{H}_{\eta} \mathcal{E}_{\eta}$. Now we can see that if a function $f$ satisfy the system:
\begin{equation}\label{define1}
	\left\{\begin{array}{ll}
	\Delta_{\eta}f=  0 \quad &\text{on}\quad \Omega^{\pm},\\
	\left\langle\nabla_{\eta} f,(0,1)\right\rangle=0 \quad&\text { on }\quad\Gamma^{\pm},\\
	f^{+}-f^{-} \circ \sigma=\gamma_1  \quad &\text {on } \quad \Gamma,\\
	\left\langle\nabla_{\eta^{+}}f^{+}-\nabla_{\eta^{-}} f^{-} \circ \sigma, n\right\rangle =\gamma_2 \quad &\text {on } \quad \Gamma.
	\end{array}\right.
	\end{equation}
 Then we can represent the relationship between the boundary data and the solution by $f=\mathcal{H}^{\partial}_{\eta}(\gamma_1,\gamma_2)$.  We call the equations \eqref{decompose1} the harmonic part, and we denote  the solution to  \eqref{decompose1} by 
 $q_{1}$. From \eqref{decompose1}, we know that $\gamma_{1}=0$ and $\gamma_{2}=M$, therefore $q_{1}=\mathcal{H}^{\partial}_{\eta}(0,M)$. At the same time, we denote the solution to  \eqref{decompose2} by $q_{2}$ . Finnaly we deduce that the solution $q$ to system \eqref{2.16} is given by $$q=q_{1}+q_{2}=\mathcal{H}^{\partial}_{\eta}(0,M)+\Delta_{\eta}^{-1}[tr((\nabla_{\eta} \partial_{t}{\eta})^2)+tr((\nabla_{\eta}(B_{0}\cdot \nabla \eta))^2)].$$
	\end{proof}

\begin{lemm}\label{nonlinear parameter}
		Via parameterization of $\eta(t,x)$, we take $\eta(t,s,x):=\eta(t,x)+sv(t,x)+o(s)$, where $o(s)$ is some infinitesimal quantity of s. Then, we apply $\partial_{s}$ to the system $\mathbf{ I}$ and  the system $\mathbf{ II}$ to get  the following  equality:
\begin{equation}
\begin{aligned}
\partial_{s} q&=\partial_{s} q_{1}+\partial_{s} q_{2}\\
&=-\Delta^{-1}_{\eta}[u\cdot\nabla,\Delta]_{\eta}\mathcal{H}^{\partial}_{\eta}(0,M)+ \mathcal{H}^{\partial}_{\eta}(\alpha(\eta ,v,\mathcal{H}^{\partial}_{\eta}(0, M)), \beta(\eta ,v,\mathcal{H}^{\partial}_{\eta}(0, M)))\\
&+\mathcal{H}^{\partial}_{\eta}(0, \partial_{s}M)-\Delta^{-1}_{\eta(t,s)}[u\cdot\nabla,\Delta]_{\eta}\Delta^{-1}_{\eta} [tr((\nabla_{\eta} \partial_{t}{\eta})^2)+tr((\nabla_{\eta}(B_{0}\cdot \nabla \eta))^2)] \\
&+ \Delta^{-1}_{\eta(t,s)} \partial_{s}[tr((\nabla_{\eta}(B_{0}\cdot \nabla \eta))^2) -tr((\nabla_{\eta} \partial_{t}{\eta})^2)] +  \mathcal{H}^{\partial}_{\eta}(\alpha(\eta ,v,q_{2}), \beta(\eta ,v,q_{2})).
\end{aligned}
 \end{equation}	
\end{lemm}
\begin{proof}
To compute $\partial_{s} q_{1}=\partial_{s}\mathcal{H}^{\partial}_{\eta}(0, M)$, we need to parameterize \eqref{decompose1}, apply $\partial_{s}$ to the system $\mathbf{ I}$  and then taking $s=0$. Thus we have
 	\begin{equation}\label{F-decompose1}
  \mathbf{I^{\prime}}:\left\{\begin{array}{ll}
 \Delta_{\eta}\partial_{s}q_{1}=-[u\cdot\nabla,\Delta]_{\eta}\mathcal{H}^{\partial}_{\eta}(0, M)  &\text{on}\quad \Omega^{\pm}\\
  \partial_{s}q^{+}_{1}- \partial_{s}q^{-}_{1} \circ \sigma=\alpha(\eta ,v,\mathcal{H}^{\partial}_{\eta}(0, M))  &\text {on } \quad \Gamma,\\
 \left\langle\nabla_{\eta^{+}} \partial_{s}q^{+}_{1}-\nabla_{\eta^{-}} \partial_{s}q^{-}_{1} \circ \sigma, n\right\rangle \\
 =\partial_{s}M+\beta(\eta ,v,\mathcal{H}^{\partial}_{\eta}(0, M))  &\text {on } \quad \Gamma,
 \end{array}\right.
 \end{equation}
where 
 \begin{align}
\alpha(\eta ,v,\mathcal{H}^{\partial}_{\eta}(0, M))=\nabla_{\!\partial_{s}\sigma} q^{-}_{1} \circ \sigma\notag,
\end{align}
\begin{align}
&\quad\beta(\eta ,v,\mathcal{H}^{\partial}_{\eta}(0, M))\\\notag&= \left\langle\nabla_{\eta^{-}} \nabla_{\partial_{s}\sigma}q^{-}_{1}-[u^{+}\cdot\nabla,\nabla]_{\eta^{+}}q^{+}_{1}+[u^{-}\cdot\nabla,\nabla]_{\eta^{-}}q^{-}_{1}, n\right\rangle\\\notag
&-\left\langle\nabla_{\eta^{+}} q^{+}_{1}-\nabla_{\eta^{-}} q^{-}_{1}, \partial_{s} n\right\rangle.
\end{align}
To find the solution, we decompose \eqref{F-decompose1} into the following three subsystems:
\begin{equation}
 \mathbf{I^{\prime}-1}: \left\{\begin{array}{ll}
\Delta_{\eta}\partial_{s}q_{11}=-[u\cdot\nabla,\Delta]_{\eta}\mathcal{H}^{\partial}_{\eta}(0, M)  &\text{on}\quad \Omega\\
\partial_{s}q^{+}_{11}- \partial_{s}q^{-}_{11} \circ \sigma=0 &\text {on } \quad \Gamma,\\
\left\langle\nabla_{\eta^{+}} \partial_{s}q^{+}_{11}-\nabla_{\eta^{-}} \partial_{s}q^{-}_{11} \circ \sigma, n\right\rangle =0  &\text {on } \quad \Gamma,
\end{array}\right.
\end{equation}
and
\begin{equation}
 \mathbf{I^{\prime}-2}: \left\{\begin{array}{ll}
\Delta_{\eta}\partial_{s}q_{12}=0  &\text{on}\quad \Omega\\
\partial_{s}q^{+}_{12}- \partial_{s}q^{-}_{12} \circ \sigma=\alpha &\text {on } \quad \Gamma,\\
\left\langle\nabla_{\eta^{+}} \partial_{s}q^{+}_{12}-\nabla_{\eta^{-}} \partial_{s}q^{-}_{12} \circ \sigma, n\right\rangle =\beta  &\text {on } \quad \Gamma,
\end{array}\right.
\end{equation}
and
\begin{equation}
  \mathbf{I^{\prime}-3}:\left\{\begin{array}{ll}
\Delta_{\eta}\partial_{s}q_{13}=0  &\text{on}\quad \Omega\\
\partial_{s}q^{+}_{13}- \partial_{s}q^{-}_{13} \circ \sigma=0 &\text {on } \quad \Gamma,\\
\left\langle\nabla_{\eta^{+}} \partial_{s}q^{+}_{13}-\nabla_{\eta^{-}} \partial_{s}q^{-}_{13} \circ \sigma, n\right\rangle =\partial_{s}M &\text {on } \quad \Gamma,
\end{array}\right.
\end{equation}
Then the solution $q$ to system \eqref{2.16} is given by 
\begin{equation}
\begin{aligned}
\partial_{s} q_{1}:&=q_{11}+q_{12}+q_{13}\\
&=-\Delta^{-1}_{\eta}[u\cdot\nabla,\Delta]_{\eta}\mathcal{H}^{\partial}_{\eta}(0,M)+ \mathcal{H}^{\partial}_{\eta}(\alpha(\eta ,v,\mathcal{H}^{\partial}_{\eta}(0, M)), \beta(\eta ,v,\mathcal{H}^{\partial}_{\eta}(0, M)))\\
&+\mathcal{H}^{\partial}_{\eta}(0, \partial_{s}M).
\end{aligned}
 \end{equation}
To compute $\partial_{s}q_{2}=\partial_{s} \Delta^{-1}_{\eta} [tr((\nabla_{\eta} \partial_{t}{\eta})^2)+tr((\nabla_{\eta}(B_{0}\cdot \nabla \eta))^2)]$, we need to parameterize \eqref{decompose2}, apply $\partial_{s}$ to the system $\mathbf{ II}$  and then taking $s=0$. Thus we have
 	\begin{equation}\label{F-decompose2}
  \mathbf{II^{\prime}}:\left\{\begin{array}{ll}
 \Delta_{\eta}\partial_{s}q_{2}=-[u\cdot\nabla,\Delta]_{\eta} \Delta^{-1}_{\eta(t,s)} [tr((\nabla_{\eta} \partial_{t}{\eta})^2)+tr((\nabla_{\eta}(B_{0}\cdot \nabla \eta))^2)] \\
 + \partial_{s}[tr((\nabla_{\eta}(B_{0}\cdot \nabla \eta))^2) -tr((\nabla_{\eta} \partial_{t}{\eta})^2)]  &\text{on}\quad \Omega^{\pm}\\
  \partial_{s}q^{+}_{2}- \partial_{s}q^{-}_{2} \circ \sigma=\alpha &\text {on } \quad \Gamma,\\
 \left\langle\nabla_{\eta^{+}} \partial_{s}q^{+}_{2}-\nabla_{\eta^{-}} \partial_{s}q^{-}_{2} \circ \sigma, n\right\rangle =\beta  &\text {on } \quad \Gamma.
 \end{array}\right.
 \end{equation}
 As before, we decompose \eqref{F-decompose2} into the following systems,
\begin{equation}
 \mathbf{II^{\prime}-1}: \left\{\begin{array}{ll}
\Delta_{\eta}\partial_{s}q_{21}=-[u\cdot\nabla,\Delta]_{\eta}\Delta^{-1}_{\eta} [tr((\nabla_{\eta} \partial_{t}{\eta})^2)+tr((\nabla_{\eta}(B_{0}\cdot \nabla \eta))^2)] \\
+ \partial_{s}[tr((\nabla_{\eta}(B_{0}\cdot \nabla \eta))^2) -tr((\nabla_{\eta} \partial_{t}{\eta})^2)]  &\text{on}\quad \Omega^{\pm}\\
\partial_{s}q^{+}_{21}- \partial_{s}q^{-}_{21} \circ \sigma=0 &\text {on } \quad \Gamma,\\
\left\langle\nabla_{\eta^{+}} \partial_{s}q^{+}_{21}-\nabla_{\eta^{-}} \partial_{s}q^{-}_{21} \circ \sigma, n\right\rangle =0  &\text {on } \quad \Gamma,
\end{array}\right.
\end{equation}
where 
\begin{equation}
\begin{aligned}
&\partial_{s}[tr((\nabla_{\eta}(B_{0}\cdot \nabla \eta))^2) -tr((\nabla_{\eta} \partial_{t}{\eta})^2)] \\
&=-2\nabla_{\!\eta}\Delta_{\eta}^{-1}\big(tr(\nabla_{\!\eta}\eta_{t}\cdot\partial_{s}D_{\eta}\eta_{t}|_{s=0})\\
&+tr(D_{\eta}(H_0\cdot\nabla\eta)\partial_{s}D_{\eta}(H_0\cdot\nabla\eta)|_{s=0})\big)
\end{aligned}
\end{equation}
and
\begin{equation}
 \mathbf{II^{\prime}-2}: \left\{\begin{array}{ll}
\Delta_{\eta}\partial_{s}q_{22}=0  &\text{on}\quad \Omega^{\pm}\\
\partial_{s}q^{+}_{22}- \partial_{s}q^{-}_{22} \circ \sigma=\alpha &\text {on } \quad \Gamma,\\
\left\langle\nabla_{\eta^{+}} \partial_{s}q^{+}_{22}-\nabla_{\eta^{-}} \partial_{s}q^{-}_{22} \circ \sigma, n\right\rangle =\beta  &\text {on } \quad \Gamma,
\end{array}\right.
\end{equation}
Then the solution $q_{2}$ to system \eqref{2.16} is given by 
\begin{equation}
\begin{aligned}
\partial_{s} q_{2}:&=q_{21}+q_{22}\\
&=-\Delta^{-1}_{\eta(t,s)}[u\cdot\nabla,\Delta]_{\eta}\Delta^{-1}_{\eta} [tr((\nabla_{\eta} \partial_{t}{\eta})^2)+tr((\nabla_{\eta}(B_{0}\cdot \nabla \eta))^2)] \\
&+ \Delta^{-1}_{\eta(t,s)} \partial_{s}[tr((\nabla_{\eta}(B_{0}\cdot \nabla \eta))^2) -tr((\nabla_{\eta} \partial_{t}{\eta})^2)] \\
&+  \mathcal{H}^{\partial}_{\eta}(\alpha(\eta ,v,q_{2}), \beta(\eta ,v,q_{2})).
\end{aligned}
 \end{equation}
\end{proof}
  
	\begin{lemm}\label{lem:nonlinear}
  Suppose $v(t)=\left.\partial_{s} \eta(t, s)\right|_{s=0}$,  $u(t)=v(t) \circ \eta(t)^{-1}$, we have 
  \begin{equation*}
  	\|D^2Z(\eta,\partial_{t}{\eta})(v,\partial_{t}{v})(v,\partial_{t}{v})\|_{s}\le  C(\|(\eta,\partial_{t}{\eta})\|_{s+3})\|(v,\partial_{t}{v})\|_{s+2}\|(v,\partial_{t}{v})\|_{s},
  \end{equation*}
  where $\|. \|_{s}$ denotes the $H^{s}$-norm on $\Omega$.
	\end{lemm}
\begin{proof}
	 By \eqref{total pressure}, we get
	\begin{align}\label{0oder}
	\quad Z(\eta,\partial_{t}{\eta})=&(H_{0} \cdot \nabla \eta) \cdot \nabla_{\eta}(H_{0} \cdot \nabla \eta)-\nabla_{\eta}q\\\notag
	=&(H_{0} \cdot \nabla \eta) \cdot \nabla_{\eta}(H_{0} \cdot \nabla \eta)-\nabla_{\eta}\mathcal{H}^{\partial}_{\eta}(0, M)\\\notag
 &-\nabla_{\eta}\Delta_{\eta}^{-1}tr((\nabla_{\eta} \partial_{t}{\eta})^2)+\nabla_{\eta}\Delta_{\eta}^{-1}tr((\nabla_{\eta}(H_{0}\cdot \nabla \eta))^2).
	\end{align}
	We need to compute $DZ(\eta,\partial_{t}{\eta})(v,\partial_{t}{v})$ firstly. To do it, we need to parameterize $\eta=\eta(t,x)$ by $\eta=\eta(t,s,x)$ which satisfies $\eta(t,s,x)|_{s=0}=\eta(t,x):=\eta$ and $\partial_{s}\eta(t,s,x)|_{s=0}=v$, i.e, we take $\eta(t,s,x)=\eta(t,x)+sv(t,x)+o(s)$ with $o(s)$ meaning some infinitesimal quantity of s. Then
taking using of lemma 3.1-3.3,  we find that 
	\begin{equation}\label{parameterization of first derivative}
	\begin{aligned}
	&\quad DZ(\eta,\partial_{t}\eta)(v,\partial_{t}{v})\\
	&=\partial_{s}Z(\eta(t,s),\partial_{t}\eta(t,s))|_{s=0}\\
 &=(H_{0} \cdot \nabla v) \cdot \nabla_{\eta}(H_{0} \cdot \nabla \eta)+(H_{0} \cdot \nabla \eta) \cdot [u\cdot\nabla,\nabla]_{\eta} (H_{0} \cdot \nabla \eta)  \\
		&+(H_{0} \cdot \nabla \eta) \cdot \nabla_{\eta}(H_{0} \cdot \nabla v) -
  [u\cdot\nabla,\nabla]_{\eta} q-\nabla_{\!\eta}\partial_s q.
 \end{aligned}
	\end{equation}
where the  term $\nabla\partial_s q$ in \eqref{parameterization of first derivative} can be estimated by using the lemma \eqref{nonlinear parameter}.

By the above computation and lemma \eqref{nonlinear parameter}, we see that $\partial_{s} q$ is zeroth order in $(v, \partial_{t}{v})$, therefore  we deduce that $D Z(\eta, \partial_{t}{\eta})(v, \partial_{t}{v})$ is a second order differential operator in $(v, \partial_{t}{v})$. Similarly we find that $D^{2} Z$ is also second order differential operator in each $(v, \partial_{t}{v})$. From this, using the usual Sobolev estimates, we find for $s>1$ that 
	 $$
	 \left\|D^{2} Z(\eta, \partial_{t}{\eta})(v, \partial_{t}{v})\right\|_{s} \leq C(\|(\eta, \partial_{t}{\eta})\|_{s+3})\|(v, \partial_{t}{v})\|_{s+2}\|(v, \partial_{t}{v})\|_{s},
	 $$
	 where $C(\|(\eta, \partial_{t}{\eta})\|_{s+3})$ is uniform for all $(\eta, \partial_{t}{\eta})$ near $(\zeta, \partial_{t}{\zeta})$ .
\end{proof}

\subsection{Proof of main theorem.}
\quad \quad We are now in a position to prove  ill-posedness for the nonlinear problem.  To see ill-posedness,  we first decompose $\chi_{n}(t)=\eta_{n}(t)-\xi(t)$ into harmonic part $h$ and an remaining part $r$ as follows:
	\begin{equation} \label{3.10}
	\chi_{n}= \nabla h+r,
	\end{equation}
	where $h$ satisfy
	\begin{equation} \label{3.11}
	\begin{array}{rll}
	\Delta h=0 & \text { in } & \Omega^{ \pm} \\
	\partial_{x_{3}} h=0 & \text { on } & \Gamma^{ \pm}  \\
	\chi_{n,3}=\partial_{x_{3}} h & \text {on } &  \Gamma,
	\end{array}
	\end{equation}
 with $\chi_{n,3}$ meaning the third component of vector $\chi_{n}$
	and $r$ satisfies
	\begin{equation}\label{3.12}
	r_{3}=0 \quad \text {on } ~~ \Gamma ~~\text {and }  ~~\Gamma^{ \pm}.
	\end{equation}
	
	To solve \eqref{3.11}, similar to Ebin \cite{E2},  we let $h=f+g$ where $f$ and $g$ are harmoic functions defined by:
 	\begin{equation} \label{3.111}
	\begin{array}{rll}
	\partial_{x_{3}}f=\partial_{x_{3}}g=0 &  \text { on } & \Gamma^{ \pm}, \\
	\partial_{x_{3}}f^{+}=\partial_{x_{3}}f^{-} & \text { on } & \Gamma,  \\
	g^{+}=g^{-} & \text {on } &  \Gamma.
	\end{array}
	\end{equation}
Thus we can see that  $g$ is the even part of $h$ in $x_{3}$ and $f$ is the odd part.  Since $f$ is harmonic and satisfies \eqref{3.111} , it is  a linear combination of functions of the following  form:
\begin{equation}\label{3.13}
	f_{j}=P a\left\{\begin{array}{ll}
	e^{i j(x_{1}+t)}(\sinh j x_{3}-\operatorname{coth} j  \cosh j x_{3}) & x_{3} \geq 0 \\
	e^{i j(x_{1}-t)}(\sinh j x_{3}+\operatorname{coth} j  \cosh j x_{3}) & x_{3}<0.
	\end{array}\right\}
	\end{equation}
Similarly, we can see that  $g$  is  a linear combination of functions of the following  form:
\begin{equation}\label{3.133}
	g_{j}=P a\left\{\begin{array}{ll}
	e^{i j(x_{1}+t)}(-\sinh j x_{3}+\operatorname{coth} j  \cosh j x_{3}) & x_{3} \geq 0 \\
	e^{i j(x_{1}-t)}(\sinh j x_{3}+\operatorname{coth} j  \cosh j x_{3}) & x_{3}<0.
	\end{array}\right\}
	\end{equation}
 
For the odd part $f$ we further decompose it into  high frequencies $\mathcal{P}$ and low frequencies $\mathcal{L}$, more precisely,  let    $\mathcal{P}$ be the linear combination of $f_{j}$ for $j \geq n$, and let $\mathcal{L}$ be the part with $j<n$. Therefore, we decompose the perturb quantity $\chi_{n}$  into fourth part:
	\begin{equation}\label{3.14}
	\chi_{n}=\nabla \mathcal{P}_n +\nabla \mathcal{L}_n+ \nabla g_{n}+r_n,
	\end{equation}

For the simplicity of notation, we ignore the subscripts $n$ of $\mathcal{P}_n,\nabla \mathcal{L}_n, \nabla g_{n}, r_n$ in the following part of this paper.
	By this definition, we verify that the summands of \eqref{3.14} are orthogonal with respect to the $L^{2}(\Omega)$:
	\begin{equation}\label{3.15}
	\begin{array}{rll}
	\int_{\Omega} \left\langle \nabla h, r\right\rangle =& \int_{\Omega} \left\langle \nabla h, \chi_{n}\right\rangle- \int_{\Omega} \left\langle \nabla h, \nabla h\right\rangle\\
	=& -\int_{\Omega} \left\langle  h, \mathrm{div} \chi_{n}\right\rangle+ \int_{\partial \Omega} h \left\langle \chi_{n},\nu\right\rangle\\
	&+ \int_{\Omega} \left\langle \Delta h, h\right\rangle
	- \int_{\partial \Omega} h \left\langle \nabla h,\nu\right\rangle\\
	=&0,
	\end{array}
	\end{equation}
where we use the fact that $\chi_{n,2}(x_2=\pm 1)=\eta_{n,3}(x_3=\pm 1)-\xi_2(x_3=\pm 1)=0$ with $\eta_3$ meaning the second component of vector $\eta$ and $\nu$ is the unit out normal vector. Meanwhile, note that $\nabla \mathcal{P}$ and $\nabla \mathcal{L}$  belong to the space $T_{\eta}M^{s}$ while $\nabla g$ is perpendicular to the tangent space  $T_{\eta}M^{s}$.  Moreover, similar to (5.19) and (5.20) in \cite{E1}, we can also show that $\nabla\mathcal{P}$ and $\nabla \mathcal{L}$ are orthogonal.

Next,  we turn  to derive an equation for each of the summands for $\chi_{n}$, to do this, we  compute $D Z$ applied to every part of $\chi_{n}$.

For the remaining part $r$, by the definition of  $r$, we deduce  $r_{3}|_{x_{3}=0}=0$,  substituting  this identity into system \eqref{2.22}, it follows that $\partial_{s} q=0$. Therefor we have  $D Z(r, \partial_{t}{r})= k \partial^{2}_{x_{2}} v-\nabla\partial_s q= k \partial^{2}_{x_{2}} v$, where $k$ is defined as follows:
\begin{equation}
	k=\left \{\begin{array}{ll}
	a^{2} & on ~~~ \Omega^{+}  \\
	b^{2} & on ~~~ \Omega^{-} .
	\end{array} \right \}
	\end{equation}
	
	To compute $D Z(\nabla f, \partial_{t}{\nabla f})$, following the computation  in section 2, it follows that : $D Z(\nabla f, \partial_{t}{\nabla f})=j^{2} \nabla f:=  j^{2} \nabla f $.  Let us define an operator $A$  on $\Omega_{ \pm}$ as follows: If $f_{j}$ is as in \eqref{3.13}, let $A \nabla f_{j}=j^{2} \nabla f_{j}$, then $D Z(\nabla f, \nabla\partial_{t}{f})= A \nabla f $. 
 
 It remains to compute $D Z(\nabla g, \partial_{t}{\nabla g})$, let $v=\nabla g$, taking use of \eqref{3.133}, we have  $\partial_{x_1} \partial_t v^{+}_{3} + \partial_{x_1} \partial_t v^{-}_{3}\circ \sigma + \frac{1}{2}(a^2\partial^{2}_{x_2}v^+_3+b^2\partial^{2}_{x_2}v^-_3\circ \sigma)=2j^{2} \partial_{x_2} g_{j}$, then $D Z(\nabla g, \nabla\partial_{t}{g})=-2 A \nabla g $. This completes the computation of $D Z$.

	For convenience, we define $Q$ to be  the double integral term,    \eqref{double integration} can  be rewritten as  
	
	\begin{equation} \label{3.16}
	\partial_{t}^{2} {\chi_{n}}=D Z(\eta, \partial_{t} {\eta})(\chi_{n}, \partial_{t}\chi_{n})+Q
	\end{equation}
	and using \eqref{3.14} we decompose $Q$ into $Q_{1}+Q_{2}+Q_{3}+Q_{4}$. Thus we have
	
	\begin{equation}\label{3.17}
	\begin{aligned}
	& \partial_{t}^{2}\nabla \mathcal{P}= A\nabla \mathcal{P}+Q_{1}, \\
	& \partial_{t}^{2}\nabla \mathcal{L}= A \nabla \mathcal{L}+Q_{2}, \\
    &\partial_{t}^{2}\nabla g=-2 A \nabla g+Q_{3},\\
    &\partial_{t}^{2} r=-k A r+ Q_{4} \text {. }
	\end{aligned}
	\end{equation}

	Now we are going to estimate the growth of the solutions to above equations  \eqref{3.17}, to do this, we  introduce some functionals which can be used to estimate its growth:
	\begin{equation}\label{3.19}
	\begin{aligned}
	E_{\mu}^{ \pm} & =\left\|A^{\frac{\mu}{2}} \partial_{t} \nabla\mathcal{P} \pm   A^{\frac{\mu+1}{2}} \nabla \mathcal{P}\right\|^{2}_{0} \\
	E_{\mu} & =E_{\mu}^{+}+E_{\mu}^{-} \\
	G & =\|\partial_{t} \nabla \mathcal{L}\|^{2}_{0}+\| A^{\frac{1}{2}} \nabla \mathcal{L}\|^{2}_{0},
	\end{aligned}
	\end{equation}
	and
	\begin{equation}\label{3.20}
	\begin{aligned}
	F=\|\partial_{t} \nabla g\|^{2}_{0}+\| A^{\frac{1}{2}} \nabla g\|^{2}_{0}+\|\partial_{t} r\|^{2}_{0}+\|k^{\frac{1}{2}} A^{\frac{1}{2}} r\|^{2}_{0},
	\end{aligned}
	\end{equation}
	where $"\| \|_{0}"$ means the $L^{2}$-norm on $\Omega$ and $"\| \|_{s}"$ denotes the $H^{s}$-norm on $\Omega$.  From the definitions of $E$ and $A$, we can directly calculate these inequalities: $E_{\mu}^{ \pm} \geq n^{2(\mu-\nu)} E_{\nu}^{ \pm}$, and $\left\|A^{\mu} \nabla \mathcal{L}\right\|_{0} \leq(n-1)^{2 \mu}\|\nabla \mathcal{L}\|_{0}$.
	
	\begin{prop} 
		For sufficiently large $n$, the set of sufficiently small $z$ such that $E_{1}^{+} \geq E_{1}^{-}, E_{1}^{+} \geq n^{3} F$, and $E_{1}^{+} \geq n^{3} G$ is invariant under the evolution defined by \eqref{3.19} and \eqref{3.20}.
	\end{prop}
	
	\begin{proof}
		From these equations \eqref{3.17} we compute the time derivatives of $E_{1}^{ \pm}, F$ and $G$ as follows:
		\begin{equation}\label{3.21}
		\begin{aligned}
		\partial_{t}E_{1}^{ \pm} & =2\left(A^{\frac{1}{2}} \partial^{2}_{t} \nabla \mathcal{P} \pm  \alpha^{\frac{1}{2}}  A \partial_{t} \nabla \mathcal{P}, A^{\frac{1}{2}} \partial_{t} \nabla \mathcal{P} \pm  \alpha^{\frac{1}{2}}  A \nabla\mathcal{P}\right) \\
		&=2\left(A^{\frac{1}{2}} [\alpha A\nabla \mathcal{P}+Q_{1}]   \pm  \alpha^{\frac{1}{2}}  A \partial_{t} \nabla\mathcal{P}, A^{\frac{1}{2}} \partial_{t} \nabla \mathcal{P} \pm \alpha^{\frac{1}{2}}  A \nabla \mathcal{P}\right) \\
		&= \pm 2 \alpha^{\frac{1}{2}}  E_{1 \frac{1}{2}}^{ \pm}+2 \left(A^{\frac{1}{2}} \partial_{t} \nabla\mathcal{P} \pm \alpha^{\frac{1}{2}}  A \nabla \mathcal{P}, A^{\frac{1}{2}} Q_{1}\right),
		\end{aligned}
		\end{equation}
		
		\begin{equation}\label{3.22}
		\begin{aligned}
		\partial_{t} G& =2(\partial_{t}\nabla  \mathcal{L}, \alpha  A \nabla \mathcal{L}+ Q_{2})+2\left( \alpha^{\frac{1}{2}} A^{\frac{1}{2}} \nabla \mathcal{L},  \alpha^{\frac{1}{2}} A^{\frac{1}{2}}\partial_{t} \nabla \mathcal{L}\right) \\
		& =4\left(\alpha^{\frac{1}{2}}A^{\frac{1}{2}} \partial_{t}\nabla  \mathcal{L}, \alpha^{\frac{1}{2}}A^{\frac{1}{2}} \nabla \mathcal{L}\right)+2\left(Q_{2}, \partial_{t}\nabla  \mathcal{L}\right) \\
		& \leq 4\alpha^{\frac{1}{2}}(n-1)\left(\partial_{t}\nabla  \mathcal{L}, \alpha^{\frac{1}{2}}A^{\frac{1}{2}} \nabla \mathcal{L}\right)+\sqrt{G}\left\|Q_{2}\right\| \\
		& \leq 2(n-1)\alpha^{\frac{1}{2}} G+\sqrt{G}\left\|Q_{2}\right\|_{0},
		\end{aligned}
		\end{equation}
		\begin{equation}\label{3.23}
		\begin{aligned}
		\partial_{t} F & =2\left(Q_{3}, \partial_{t} \nabla g \right)+2\left(Q_{4}, \partial_{t} r \right)\\
		&\leq \sqrt{F}(\left\|Q_{3}\right\|_{0}+\left\|Q_{4}\right\|_{0}),
		\end{aligned}
		\end{equation} 
		
		Using the estimate of the nonlinear term in lemma 3.6,  we find that:
		\begin{equation}\label{3.24}
		\|Q\|_{1} \leq K\|(\chi_{n}, \partial_{t} \chi_{n})\|_{3}\|(\chi_{n}, \partial_{t} \chi_{n})\|_{1},
		\end{equation}
		where  $K$ is a generic constant which is independent of $E_{\mu}^{ \pm}, F, G$ and $n$.
		We can assume that $\|(\chi_{n}, \partial_{t} \chi_{n})\|_{3} $ tend to zero as $n\rightarrow \infty$, thus we have 
		\begin{equation}\label{3.25}
		\begin{aligned}
		\|Q\|_{1} \leq \epsilon_{n}\|(\chi_{n}, \partial_{t} \chi_{n})\|_{1} 
		\end{aligned}
		\end{equation}
		where $ \epsilon_{\mathfrak{n}}$ denotes any sequence of positive numbers whose limit is zero. 
		
		Now we conclude the proof by showing that if the inequalities of the proposition hold at a given time, then so do the time derivatives of the inequalities.
		\begin{equation}\label{3.26}
		\begin{aligned}
		\partial_{t}E_{1}^{+}-	\partial_{t}E_{1}^{-} & =2 \alpha^{\frac{1}{2}}E_{1 \frac{1}{2}}+2\left(\alpha^{\frac{1}{2}}  A \nabla \mathcal{P}, A^{\frac{1}{2}} Q_{1}\right) \\
		& \geq 2 \alpha^{\frac{1}{2}}E_{1 \frac{1}{2}}- \sqrt{E_{1}}\left\|Q_{1}\right\|_{1} \\
		& \geq 2\alpha^{\frac{1}{2}}E_{1 \frac{1}{2}}- \epsilon_{n} E_{1} \\
		& \geq 2 \alpha^{\frac{1}{2}}E_{1 \frac{1}{2}}^{+}\left(1-\frac{\epsilon_{n}}{2n\alpha^{\frac{1}{2}}}\right) \\
		& \geq 0 \text { for } n \text { large. }
		\end{aligned}
		\end{equation}
		and 
		\begin{equation}\label{3.27}
		\begin{aligned}
		\partial_{t}E_{1}^{+}-n^{3}\partial_{t}F & \geq 2\alpha^{\frac{1}{2}} E_{1 \frac{1}{2}}^{+}-n^{3} F-\sqrt{E_{1}}\left\|Q_{1}\right\|_{1}-n^{3} \sqrt{F}\left\|Q_{3}\right\|_{0}\\
		& \geq 2 \alpha^{\frac{1}{2}}E_{1 \frac{1}{2}}^{+}-E_{1}^{+}-\epsilon_{n} E_{1}-n^{\frac{3}{2}} \sqrt{E_{1}} \epsilon_{n} \sqrt{E_{0}} \\
		& \geq 2 \alpha^{\frac{1}{2}}E_{1 \frac{1}{2}}^{+}\left(1-\frac{1}{\alpha^{\frac{1}{2}}n}-\epsilon_{n}\left(\frac{1}{\alpha^{\frac{1}{2}}n}+\frac{1}{\alpha^{\frac{1}{2}}n^{\frac{1}{2}}}\right)\right) \\
		& \geq 0 \text { for } n \text { large. }
		\end{aligned}
		\end{equation}
		
		Finally from \eqref{3.22}, we derive 
		\begin{equation}\label{3.29}
		\begin{aligned}
		\partial_{t}E_{1}^{+}-n^{3} \partial_{t}G & \geq 2 \alpha^{\frac{1}{2}}E_{1 \frac{1}{2}}^{+}-\sqrt{E_{1}^{+}} \epsilon_{n} \sqrt{E_{0}}\\
		&-2(n-1) E_{1}^{+}-n^{\frac{3}{2}} \sqrt{E_{1}^{+}} \epsilon_{n} E_{0} \\
		& \geq E_{1 \frac{1}{2}}^{+}\left(2-\epsilon_{n} \frac{1}{\alpha^{\frac{1}{2}}n^{2}}-2 \frac{n-1}{\alpha^{\frac{1}{2}}n}-\epsilon_{n}\frac{1}{\alpha^{\frac{1}{2}}} n^{\frac{3}{2}-\frac{1}{2}-\frac{3}{2}}\right)\\
		& \geq 0 \text { for } n \text { large. }
		\end{aligned}
		\end{equation}
		Thus we have  proved Proposition 2.
	\end{proof} 
	
	\begin{coro}
		For sufficiently large $n$, $E_{1}^{+}(t) \geq E_{1}^{+}(0) e^{ n t}$, and therefore \\
		 $\left\|\nabla \mathcal{P}_{n}(t)\right\|_{2} \geq e^{n t} E_{1}^{+}(0)$.
	\end{coro}
	
	\begin{proof}
		
		Using \eqref{3.21} we find
		\begin{equation}\label{3.30}
		\begin{aligned}
		\partial_{t} E_{1}^{+} & \geq 2 E_{1 \frac{1}{2}}^{+}-\sqrt{E_{1}}\left\|Q_{1}\right\|_{1} \\
		& \geq 2 E_{1 \frac{1}{2}}^{+}-E_{1} \epsilon_{n} \\
		& \geq E_{1 \frac{1}{2}}^{+}
		\end{aligned}
		\end{equation}
		
		Therefore for large $n$,
		\begin{equation}\label{3.31}
	\partial_{t} E_{1}^{+} \geq n  E_{1}^{+},
		\end{equation}
		
		and the corollary follows by integrating this inequality.
	\end{proof}
	Now, we are going to  prove the ill-posedness of the Kelvin-Helmholtz problem. We prove it by contradiction. Suppose that the system \eqref{1.15} is well posed, if we have following  initial data:
	\begin{equation}\label{3.32}
	\begin{aligned}
	& \eta_{n}(0)=\xi(0) \\
	& \partial_{t} \eta_{n}(0)=\partial_{t} \xi(0)+e^{-\sqrt{n}}\left(g_{n}, 0 ,f_{n}\right),
	\end{aligned}
	\end{equation}
	where $\left(g_{n}, f_{n}\right)$ is as follow:
	\begin{equation*}
	 \left(g_{n}, f_{n}\right)= P_{a}\left\{\begin{array}{ll}
	e^{i k (x_{1}+t)}\left(W(x_{3}),0 ,V(x_{3}) \right) & x_{3} \geq 0 \\
	e^{i k (x_{1}-t)}\left(W(x_{3}),0 ,V(x_{3}) \right)& x_{3}<0.
	\end{array}\right\}
	\end{equation*}
From this, we can see that  $\left(\eta_{n}(0), \partial_{t} \eta_{n}(0)\right)$ converges to $(\xi(0), \partial_{t} \xi (0))$ in $C^{\infty}$, then there exists a uniform time interval, $[0, \bar{T}]$, on which all $\eta_{n}(t)$ were defined and would converge to $\xi(t)$. Thus  it would go to zero on $[0, T]$ as $n \rightarrow \infty$. However  as previous, we decompose $\chi_{n}(t)=\eta_{n}(t)-\xi(t)$ into $\nabla \mathcal{P}+ \nabla \mathcal{L}+ \nabla g+r$,   by using  Proposition 2 and its Corollary 1, we have $\left\|\nabla \mathcal{P}(t)\right\|_{2} \geq K e^{2\alpha^{\frac{1}{2}}n t} e^{-\sqrt{n}}$, so for any $t>0,\left\|\nabla\mathcal{P}(t)\right\|_{2} \rightarrow \infty$ as $n \rightarrow \infty$ and and therefore $\left\|\chi_{n}(t)\right\|_{2} \rightarrow \infty$. This contradiction proves Theorem \eqref{theorem} and the Kelvin-Helmholtz problem is not well-posed.

\section{Acknowledgements}
The work of B. Xie is supported by NSF-China under grant number 11901207, 12326355, Guangzhou Science and Technology Project 2023A04J1315 and the National Key Program  of China (2021YFA1002900). 
	\bibliographystyle{amsplain}

\end{document}